\documentclass[10pt]{article}%
\usepackage{graphicx}
\usepackage{amsmath}
\usepackage{latexsym}
\usepackage{amsfonts}
\usepackage{amssymb}%
\newtheorem{theorem}{Theorem}[section]
\newtheorem{proposition}[theorem]{Proposition}
\newtheorem{lemma}[theorem]{Lemma}
\newtheorem{definition}[theorem]{Definition}
\newtheorem{corollary}[theorem]{Corollary}
\newtheorem{remark}[theorem]{Remark}

\newtheorem{notation}[theorem]{Notation}
\newenvironment{proof}[1][Proof]{\textbf{#1.} }{\ \rule{0.5em}{0.5em}}
\begin{document}

\title{Stein's lemma, Malliavin calculus, and tail bounds, with application to
polymer fluctuation exponent}
\author{Frederi G. Viens$^{1,}$\thanks{Author's reserach partially supported by NSF
grant 0606615}$\vspace*{0.1in}$\\$^{1}$Dept. Statistics and Dept. Mathematics\\Purdue University\vspace{0.02in}\\150 N. University St.\\West Lafayette, IN 47907-2067, USA;\\viens@purdue.edu}
\maketitle

\begin{abstract}
We consider a random variable $X$ satisfying almost-sure conditions involving
$G:=\left\langle DX,-DL^{-1}X\right\rangle $ where $DX$ is $X$'s Malliavin
derivative and $L^{-1}$ is the inverse Ornstein-Uhlenbeck operator. A lower-
(resp. upper-) bound condition on $G$ is proved to imply a Gaussian-type lower
(resp. upper) bound on the tail $\mathbf{P}\left[  X>z\right]  $. Bounds of
other natures are also given. A key ingredient is the use of Stein's lemma,
including the explicit form of the solution of Stein's equation relative to
the function $\mathbf{1}_{x>z}$, and its relation to $G$. Another set of
comparable results is established, without the use of Stein's lemma, using
instead a formula for the density of a random variable based on $G$, recently
devised by the author and Ivan Nourdin. As an application, via a Mehler-type
formula for $G$, we show that the Brownian polymer in a Gaussian environment
which is white-noise in time and positively correlated in space has deviations
of Gaussian type and a fluctuation exponent $\chi=1/2$. We also show this
exponent remains $1/2$ after a non-linear transformation of the polymer's Hamiltonian.

\end{abstract}

\thispagestyle{myheadings}

\vspace{0.2in}


\textbf{AMS 2000 MSC codes: }primary 60H07; secondary 60G15, 60K37, 82D60

\section{Introduction}

\subsection{Background and context}

Ivan Nourdin and Giovanni Peccati have recently made a long-awaited connection
between Stein's lemma and the Malliavin calculus: see \cite{NP}, and also
\cite{NPberry}. Our article uses crucial basic elements from their work, to
investigate the behavior of square-integrable random variables whose Wiener
chaos expansions are not finite. Specifically we devise conditions under which
the tail of a random variable is bounded below by Gaussian tails, by using
Stein's lemma and the Malliavin calculus. Our article also derives similar
lower bounds by way of a new formula for the density of a random variable,
established in \cite{NV}, which uses Malliavin calculus, but not Stein's
lemma. Tail upper bounds are also derived, using both methods.

Stein's lemma has been used in the past for Gaussian upper bounds, e.g. in
\cite{ChatExch} in the context of exchangeable pairs. Malliavin derivatives
have been invoked for similar upper bounds in \cite{VVJFA}. In the current
paper, the combination of these two tools yields a novel criterion for a
Gaussian tail lower bound. We borrow a main idea from Nourdin and Peccati
\cite{NP}, and also from \cite{NV}: to understand a random variable $Z$ which
is measurable with respect to a Gaussian field $W$, it is fruitful to consider
the random variable
\[
G:={{\langle DZ,-DL^{-1}Z\rangle,}}%
\]
{{where }}$D$ is the Malliavin derivative relative to $W$, $\left\langle
\cdot,\cdot\right\rangle $ is the inner product in the canonical Hilbert space
$H$ of $W$, and $L$ is the Ornstein-Uhlenbeck operator. Details on $D$, $H$,
$L$, and $G$, will be given below.

The function $g\left(  z\right)  =\mathbf{E}\left[  G|Z=z\right]  $ has
already been used to good effect in the density formula discovered in
\cite{NV}; this formula implied new lower bounds on the densities of some
Gaussian processes' suprema. The article \cite{NPV}, in preparation, contains
some very sharp Gaussian supremum density formulas, also based on $g$. These
results are made possible by fully using the Gaussian property, and in
particular by exploiting both upper and lower bounds on the process's
covariance. The authors of \cite{NV} noted that, if $Z$ has a density and an
upper bound is assumed on $G$, in the absence of any other assumption on how
$Z$ is related to the underlying Gaussian process $W$, then $Z$'s tail is
sub-Gaussian. On the other hand, the authors of \cite{NV} tried to discard any
upper bound assumption, and assume instead that $G$ was bounded below, to see
if they could derive a Gaussian lower bound on $Z$'s tail; they succeeded in
this task, but only partially, as they had to impose some additional
conditions on $Z$'s function $g$, which are of upper-bound type, and which may
not be easy to verify in practice.

The techniques used in \cite{NV} are well adapted to studying densities of
random variables under simultaneous lower and upper bound assumptions, but
less so under single-sided assumptions. The point of the current paper is to
show that, while the quantitative study of densities via the Malliavin
calculus seems to require two-sided assumptions as in \cite{NPV} and
\cite{NV}, single-sided assumptions on $G$ are in essence sufficient to obtain
single sided bounds on tails of random variables, and there are two strategies
to this end: Nourdin and Peccati's connection between Malliavin calculus and
Stein's lemma, and exploiting the Malliavin-calculus-based density formula in
\cite{NV}.

The key new component in our work, relative to the first strategy, may be
characterized by saying that, in addition to a systematic exploitation of the
Stein-lemma--Malliavin-calculus connection (via Lemma \ref{lemkey} below), we
carefully analyze the behavior of solutions of the so-called Stein equation,
and use them profitably, rather than simply use the fact that there exist
bounded solutions with bounded derivatives. The novelty in our second strategy
is simply to note that the difficulties inherent to using the density formula
of \cite{NV} with only one-sided assumptions, tend to dissappear when one
passes to tail formulas.\bigskip

Our work follows in the footsteps of Nourdin and Peccati's. One major
difference between our work and their's, and indeed between ours and the main
use of Stein's method since its inception in \cite{Stein72} to the most recent
results (see \cite{ChatAoP}, \cite{ChenShao}, \cite{RinottRotar}, and
references therein) is that Stein's method is typically concerned with
convergence to the normal distribution while we are only interested in rough
bounds of Gaussian or other types for single random variables (not sequences),
without imposing conditions which would lead to normal or any other
convergence. As an exception to this statement, \cite{NP} implies that a bound
on the variance of a single $G$ has clear implications for the distance from
$Z$'s distribution to the normal law [see for instance Remark 3.6 therein];
Nourdin and Peccati in \cite{NP} did not make systematic use of this result,
because their motivations did not require it.

One other difference between our work and theirs is that we do not consider
the case of a single Wiener chaos. This last specificity of our work (see
however Remark 3.8 in \cite{NP}), that we systematically consider random
variables with infinitely many non-zero Wiener chaos components, comes from
the application which we also consider in this article, to the so-called
\emph{fluctuation exponent} $\chi$ of a polymer in a random environment.
Details on this application, where we show that $\chi=1/2$ for a certain class
of environments, are in Section \ref{FLUCTU}. There is a more fundamental
obstacle to seeking upper or lower Gaussian tail bounds on an r.v. in a
\emph{single} Wiener chaos: unlike convergence results for \emph{sequences} of
r.v.'s, such as \cite{NOL}, a single $q$th chaos r.v. has a tail of order
$\exp\left(  -\left(  x/c\right)  ^{2/q}\right)  $ (see \cite{Bo}), it never
has a Gaussian behavior; our lower-bound results below (e.g. Theorem
\ref{thm12} Point 3) does apply to such an r.v., but the result cannot be sharp.

\subsection{Summary of results}

We now describe our main theoretical results. All stochastic analytic concepts
used in this introduction are described in Section \ref{PRELIM}. Let $W$ be an
isonormal Gaussian process relative to a Hilbert space $H=L^{2}\left(
T,\mathcal{B},\mu\right)  $ (for instance if $W$ is the Wiener process on
$[0,1]$, then $T=[0,1]$ and $\mu$ is the Lebesgue measure). The norm and inner
products in $H$ are denoted by $\left\Vert \cdot\right\Vert $ and
$\left\langle \cdot;\cdot\right\rangle $. Let $L^{2}\left(  \Omega\right)  $
be the set of all random variables which are square-integrable and measurable
with respect to $W$. Let $D$ be the Malliavin derivative with respect to $W$
(see Paul Malliavin's or David Nualart's texts \cite{M}, \cite{Nbook}). Thus
$DX$ is a random element in $L^{2}\left(  \Omega\right)  $ with values in the
Hilbert space $H$. The set of all $X\in L^{2}\left(  \Omega\right)  $ such
that $\left\Vert DX\right\Vert \in L^{2}\left(  \Omega\right)  $ is called
$\mathbf{D}^{1,2}$. Let $\bar{\Phi}$ be the tail of the standard normal
distribution
\[
\bar{\Phi}\left(  u\right)  :=\int_{u}^{\infty}e^{-x^{2}/2}dx/\sqrt{2\pi}.
\]
The following result, described in \cite{VVJFA} as an elementary consequence
of a classical stochastic analytic inequality found for instance in
\"{U}st\"{u}nel's textbook \cite[Theorem 9.1.1]{U}, makes use of a condition
based solely on the Malliavin derivative of a given r.v. to guarantee that its
tail is bounded above by a Gaussian tail.

\begin{proposition}
\label{fundu}For any $X\in\mathbf{D}^{1,2}$, if $\left\Vert DX\right\Vert $ is
bounded almost surely by $1$, then $X$ is a standard sub-Gaussian random
variable, in the sense that $\mathbf{P}\left[  \left\vert X-\mathbf{E}\left[
X\right]  \right\vert >u\right]  \leq2e^{-u^{2}/2}$.
\end{proposition}

\begin{remark}
The value $1$ in this proposition, and indeed in many places in this paper,
has the role of a dispersion coefficient. Since the Malliavin derivative $D$
is linear, the above proposition implies that for any $X\in\mathbf{D}^{1,2}$
such that $\left\Vert DX\right\Vert \leq\sigma$ almost surely, then
$\mathbf{P}\left[  \left\vert X-\mathbf{E}\left[  X\right]  \right\vert
>u\right]  \leq2e^{-u^{2}/(2\sigma^{2})}$. This trivial normalization argument
can be used throughout this paper, because our hypotheses are always based on
linear operators such as $D$. We use this argument in our application in
Section \ref{FLUCTU}.
\end{remark}

The question of whether a lower bound on $\left\Vert DX\right\Vert ^{2}$ gives
rise to an inequality in the opposite direction as in the above proposition
arises naturally. However, we were unable to find any proof of such a result.
Instead, after reading Eulalia Nualart's article \cite{Ncras} where she finds
a class of lower bounds by considering exponential moments on the divergence
(Skorohod integral) of a covering vector field of $X$, we were inspired to
look for other Malliavin calculus operations on $X$ which would yield a
Gaussian lower bound on $X$'s tail. We turned to the quantity $G=\left\langle
DX;-DL^{-1}X\right\rangle $, identified in \cite{NP}, and used profitably in
\cite{NPV} and \cite{NV}. Here $L^{-1}$, the inverse of the so-called
\emph{Ornstein-Uhlenbeck} operator, is defined in Section \ref{PRELIM}. This
article's first theoretical result is that a lower (resp. upper) bound on $G$
can yield a lower (resp. upper) bound similar to the upper bound in
Proposition \ref{fundu}. For instance, summarizing the combination of some
consequences of our results and Proposition \ref{fundu}, we have the following.

\begin{theorem}
\label{thm12}Let $X$ be a random variable in $\mathbf{D}^{1,2}$. Let
$G:=\left\langle DX;-DL^{-1}X\right\rangle $.

\begin{enumerate}
\item If $G\geq1$ almost surely, then%
\[
Var\left[  X\right]  \geq K_{u}:=\frac{1}{\pi^{2}}\left(  2\sqrt{1+2\sqrt
{2\pi}}-1\right)  ^{2}\simeq0.21367.
\]

\item If $G\geq1$ almost surely, and if for some $c>2$, $\mathbf{E}\left[
X^{c}\right]  <\infty$, then%
\begin{equation}
\limsup_{z\rightarrow\infty}\mathbf{P}\left[  X-\mathbf{E}\left[  X\right]
>z\right]  /\bar{\Phi}\left(  z\right)  \geq\frac{c-2}{c}.\label{supergauss}%
\end{equation}

\item If $G\geq1$ almost surely, and if there exist $c^{\prime}<1$ and
$z_{0}>0$, such that and $G\leq c^{\prime}X^{2}$ almost surely when $X\geq
z_{0}$, then for $z>z_{0}$,%
\[
\mathbf{P}\left[  X-\mathbf{E}\left[  X\right]  >z\right]  \geq\frac
{1}{2c^{\prime}+1}\bar{\Phi}\left(  z\right)
\]

\item If $G\leq1$ almost surely, then for every $z>0$%
\begin{equation}
\mathbf{P}\left[  X-\mathbf{E}\left[  X\right]  >z\right]  \leq\left(
1+\frac{1}{z^{2}}\right)  \bar{\Phi}\left(  z\right)  .\label{subgauss2}%
\end{equation}

\item If $\left\Vert DX\right\Vert ^{2}\leq1$ almost surely, then $Var\left[
X\right]  \leq\left(  \pi/2\right)  ^{2}$ and for $z>0$,%
\begin{equation}
\mathbf{P}\left[  X-\mathbf{E}\left[  X\right]  >z\right]  \leq e^{-z^{2}%
/2}\label{subgauss}%
\end{equation}

\end{enumerate}
\end{theorem}

\begin{remark}
Point 1 in this theorem is Corollary \ref{thmlbcor} Point 1. Point 2 here
comes from Corollary \ref{thmlbcor} Point 3. Point 3 here follows from
Corollary \ref{corA} Point 1. Point 4 is from Theorem \ref{thmlb}. Inequality
(\ref{subgauss}) in Point 5 here is equivalent to Proposition \ref{fundu}. The
variance upper bound in Point 5 here follows from \cite[Theorem 9.2.3 part
(iii)]{U}. Other, non-Gaussian comparisons are also obtained in this article:
see Corollary \ref{corA}.
\end{remark}

The results in Theorem \ref{thm12} point to basic properties of the Malliavin
derivative and Ornstein-Uhlenbeck operator when investigating tail behavior of
random variables. The importance of the relation of $G$ to the value $1$ was
already noticed in \cite[Theorem 3.1]{NP} where its $L^{2}$-convergence to $1$
for a sequence of r.v.'s was a basic building block for convergence to the
standard normal distribution. Here we show what can still be asserted when the
condition is significantly relaxed. An attempt was made to prove a version of
the theorem above in \cite[Section 4]{NV}; here we significantly improve that
work by: (i) removing the unwieldy upper bound conditions made in
\cite[Theorem 4.2]{NV} to prove lower bound results therein; and (ii)
improving the upper bound in \cite[Theorem 4.1]{NV} while using a weaker
hypothesis.\bigskip

Our results should have applications in any area of pure or applied
probability where Malliavin derivatives are readily expressed. In fact,
Nourdin and Peccati \cite[Remark 1.4, point 4]{NP} already hint that $G$ is
not always as intractable as one may fear. We present such an application in
this article, in which the deviations of random polymer models in some random
media are estimated, and its \emph{fluctuation exponent} is calculated to be
$\chi=1/2$, a result which we prove to be robust to non-linear changes in the
polymer's Hamiltonian.\bigskip

The structure of this article is as follows. Section \ref{PRELIM} presents all
necessary background information from the theory of Wiener chaos and the
Malliavin calculus needed to understand our statements and proofs. Section
\ref{TOOLS} recalls Stein's lemma and equation, presents the way it will be
used in this article, and recalls the density representation results from
\cite{NV}. Section \ref{MAIN} states and proves our main lower and upper bound
results. Section \ref{FLUCTU} gives a construction of continuous random
polymers in Gaussian environments, and states and proves the estimates on its
deviations and its fluctuation exponent under Gaussian and non-Gaussian
Hamiltonians, when the Gaussian environment has infinite-range correlations.
Several interesting open questions are described in this section as well.
Section 6, the Appendix, contains the proofs of some lemmas.\bigskip

\begin{description}
\item[Acknowledgements] We wish to thank Ivan Nourdin and Giovanni Peccati for
discussing their work on Stein's method with us, Eulalia Nualart for
encouraging us to study the question of lower bounds on tails of random
variables via the Malliavin Calculus, and Samy Tindel for help with the
concept of polymer fluctuation exponents.
\end{description}

\section{Preliminaries: Wiener chaos and Malliavin calculus\label{PRELIM}}

For a complete treatment of this topic, we refer the reader to David Nualart's
textbook \cite{Nbook}.

We use an abstract Wiener space given by an \emph{isonormal Gaussian process
}$W$: it is defined as a Gaussian field $W$ on a Hilbert space $H=L^{2}\left(
T,\mathcal{B},\mu\right)  $ where $\mu$ is a $\sigma$-finite measure that is
either discrete or without atoms, and the covariance of $W$ coincides with the
inner product in $H$. This forces $W$ to be linear on $H$; consequently, it
can be interpreted as an abstract Wiener integral. For instance, if $T=[0,1]$
and $\mu$ is the Lebesgue measure, then $W\left(  f\right)  $ represents the
usual Wiener stochastic integral $\int_{0}^{1}f\left(  s\right)  dW\left(
s\right)  $ of a square-integrable non-random function $f$ with respect to a
Wiener process also denoted by $W$; i.e. we confuse the notation $W\left(
t\right)  $ and $W\left(  \mathbf{1}_{[0,t]}\right)  $. In general for
$\{f_{i}:i=1,\cdots,n\}\in H^{n}$, $(W\left(  f_{i}\right)  :i=1,\cdots,n)$ is
a centered Gaussian vector, with covariance matrix given by $\sigma_{i,j}%
^{2}=\left\langle f_{i};f_{j}\right\rangle $. The set $\mathcal{H}_{1}$ of all
Wiener integrals $W\left(  f\right)  $ when $f$ ranges over all of $H$ is
called the first Wiener chaos of $W$. To construct higher-order chaoses, one
may for example use iterated It\^{o} integration in the case of standard
Brownian motion, where $H=L^{2}\left[  0,1\right]  $. If we denote
$I_{0}\left(  f\right)  =f$ for any non-random constant $f$, then for any
integer $n\geq1$ and any symmetric function $f\in H^{n}$, we let%
\[
I_{n}\left(  f\right)  :=n!\int_{0}^{1}\int_{0}^{s_{1}}\cdots\int_{0}%
^{s_{n-1}}f\left(  s_{1},s_{2},\cdots,s_{n}\right)  dW\left(  s_{n}\right)
\cdots dW\left(  s_{2}\right)  dW\left(  s_{1}\right)  .
\]
This is the $n$th iterated Wiener integral of $f$ w.r.t. $W$.

\begin{definition}
The set $\mathcal{H}_{n}:=\left\{  I_{n}\left(  f\right)  :f\in H^{n}\right\}
$ is the $n$th Wiener chaos of $W$.
\end{definition}

We refer to \cite[Section 1.2]{Nbook} for the general definition of $I_{n}$
and $\mathcal{H}_{n}$ when $W$ is a more general isonormal Gaussian process.

\begin{proposition}
\label{orthog}$L^{2}\left(  \Omega\right)  $ is the direct sum -- with respect
to the inner product defined by expectations of products of r.v.'s -- of all
the Wiener chaoses. Specifically for any $X\in L^{2}\left(  \Omega\right)  $,
there exists a sequence of non-random symmetric functions $f_{n}\in H^{n}$
with $\sum_{n=0}^{\infty}\left\Vert f_{n}\right\Vert _{H^{n}}^{2}<\infty$ such
that $X=\sum_{n=0}^{\infty}I_{n}\left(  f_{n}\right)  $. Moreover
$\mathbf{E}\left[  X\right]  =f_{0}=I_{0}\left(  f_{0}\right)  $ and
$\mathbf{E}\left[  I_{n}\left(  f_{n}\right)  \right]  =0$ for all $n\geq1$,
and $\mathbf{E}\left[  I_{n}\left(  f_{n}\right)  I_{m}\left(  g_{m}\right)
\right]  =\delta_{m,n}n!\left\langle f_{n},g_{n}\right\rangle _{H^{n}}$ where
$\delta_{m,n}$ equals $0$ if $m\neq n$ and $1$ if $m=n$. In particular
$\mathbf{E}\left[  X^{2}\right]  =\sum_{n=0}^{\infty}n!\left\Vert
f_{n}\right\Vert _{H^{n}}^{2}$.
\end{proposition}

The Malliavin derivative operator is usually constructed via an extension
starting from so-called simple random variables which are differentiable
functions of finite-dimensional vectors from the Gaussian space $\mathcal{H}%
_{1}$. The reader can consult Nualart's textbook \cite{Nbook}. We recall the
properties which are of use to us herein.

\begin{enumerate}
\item The Malliavin derivative operator $D$ is defined from $\mathcal{H}_{1}$
into $H$ by the formula: for all $r\in T$,%
\[
D_{r}W\left(  f\right)  =f\left(  r\right)  .
\]
The Malliavin derivative of a non-random constant is zero. For any
$m$-dimensional Gaussian vector $G=\left(  G_{i}\right)  _{i=1}^{m}=\left(
I_{1}\left(  g_{i}\right)  \right)  _{i=1}^{m}\in\left(  \mathcal{H}%
_{1}\right)  ^{m}$, for any $F\in C^{1}\left(  \mathbf{R}^{m}\right)  $ such
that $X=F\left(  G\right)  \in L^{2}\left(  \Omega\right)  $, we have
$D_{r}X=\sum_{i=1}^{m}\frac{\partial F}{\partial x_{i}}\left(  G\right)
g_{i}\left(  r\right)  $.

\item The Malliavin derivative of an $n$th Wiener chaos r.v. is particularly
simple. Let $X_{n}\in\mathcal{H}_{n}$, i.e. let $f_{n}$ be a symmetric
function in $H^{n}$ and $X_{n}=I_{n}\left(  f_{n}\right)  $. Then%
\begin{equation}
D_{r}X=D_{r}I_{n}\left(  f_{n}\right)  =nI_{n-1}\left(  f_{n}\left(
r,\cdot\right)  \right)  .\label{DInfn}%
\end{equation}
The Malliavin derivative being linear, this extends immediately to any random
variable $X$ in $L^{2}\left(  \Omega\right)  $ by writing $X$ as its Wiener
chaos expansion $\sum_{n=0}^{\infty}I_{n}\left(  f_{n}\right)  $, which means
that, using the covariance formulas in Proposition \ref{orthog}, $DX\in
L^{2}\left(  \Omega\times T\right)  $ if and only if%
\begin{equation}
\mathbf{E}\left[  \left\Vert DX\right\Vert ^{2}\right]  :=\sum_{n=1}^{\infty
}n\ n!\left\Vert f_{n}\right\Vert ^{2}<\infty.\label{D12}%
\end{equation}
The set of all $X\in L^{2}\left(  \Omega\right)  $ such that $DX\in
L^{2}\left(  \Omega\times T\right)  $ is denoted by $\mathbf{D}^{1,2}$.
\end{enumerate}

\begin{remark}
\label{remgenchain}The general chain rule of point 1 above generalizes to
$D\left[  h\left(  X\right)  \right]  =h^{\prime}\left(  X\right)  DX$ for any
$X\in\mathbf{D}^{1,2}$ such that $X$ has a density, and any function $h$ which
is continuous and piecewise differentiable with a bounded derivative. This is
an immediate consequence of \cite[Proposition 1.2.3]{Nbook}.
\end{remark}

In the special case of the standard Wiener space ($H=L^{2}[0,1]$) we have the
Clark-Ocone representation formula \cite[Proposition 1.3.5]{Nbook}.

\begin{proposition}
For any $X\in\mathbf{D}^{1,2}$,%
\begin{equation}
X=\mathbf{E}\left[  X\right]  +\int_{0}^{1}\mathbf{E}\left[  D_{s}%
X|\mathcal{F}_{s}\right]  dW\left(  s\right)  . \label{clarkocone}%
\end{equation}

\end{proposition}

\section{Tools: using Stein's lemma and Malliavin derivatives\label{TOOLS}}

\subsection{Stein's lemma and equation\label{STEIN}}

The version of Stein's lemma which we use can be found in \cite{NP}. Let $Z$
be a standard normal random variable and $\bar{\Phi}\left(  z\right)
=\mathbf{P}\left[  Z>z\right]  $ its tail. Let $h$ be a measurable function of
one real variable. Stein's equation poses the following question: to find a
continuous and piecewise differentiable function $f$ such that for all
$x\in\mathbf{R}$ where $f^{\prime}$ exists,%
\begin{equation}
h\left(  x\right)  -\mathbf{E}\left[  h\left(  Z\right)  \right]  =f^{\prime
}\left(  x\right)  -xf\left(  x\right)  .\label{stein}%
\end{equation}
The precise form of the solution to this differential equation, given in the
next lemma, was derived in Stein's original work \cite{Stein72}; a recent
usage is found in equalities (1.5), (2,20), and (2.21) in \cite{NPberry}.

\begin{lemma}
\label{Steinlem}Fix $z\in\mathbf{R}$. Let $h=1_{(-\infty,z]}$. Then Stein's
equation (\ref{stein}) has at least one solution $f$ satisfying $\left\Vert
f^{\prime}\right\Vert _{\infty}:=\sup_{x\in\mathbf{R}}\left\vert f^{\prime
}\left(  x\right)  \right\vert \leq1$. One such solution is the following:

\begin{itemize}
\item for $x\leq z$, $f\left(  x\right)  =\sqrt{2\pi}e^{x^{2}/2}\left(
1-\bar{\Phi}\left(  x\right)  \right)  \bar{\Phi}\left(  z\right)  ,$

\item for $x>z$, $f\left(  x\right)  =\sqrt{2\pi}e^{x^{2}/2}\left(
1-\bar{\Phi}\left(  z\right)  \right)  \bar{\Phi}\left(  x\right)  $.
\end{itemize}
\end{lemma}

\begin{corollary}
\label{Steincor}Let $X\in L^{2}\left(  \Omega\right)  $. Setting $x=X$ in
Stein's equation (\ref{stein}) and taking expectations we get%
\[
\mathbf{P}\left[  X>z\right]  =\bar{\Phi}\left(  z\right)  -\mathbf{E}\left[
f^{\prime}\left(  X\right)  \right]  +\mathbf{E}\left[  Xf\left(  X\right)
\right]  .
\]

\end{corollary}

The next section gives tools which will allow us to combine this corollary
with estimates of the random variable $G=\left\langle DX;-DL^{-1}%
X\right\rangle $ in order to get tail bounds. It will also show that $G$ can
be used, as in \cite{NV}, to express the density of $X$ without using Stein's lemma.

\subsection{Malliavin derivative tools\label{PaulTools}}

\begin{definition}
The Ornstein-Uhlenbeck operator $L$ is defined as follows. Let $X=\sum
_{n=1}^{\infty}I_{n}\left(  f_{n}\right)  $ be a centered r.v. in
$L^{2}\left(  \Omega\right)  $. If $\sum_{n=1}^{\infty}n^{2}n!\left\vert
f_{n}\right\vert ^{2}<\infty$, then we define a new random variable $LX$ in
$L^{2}\left(  \Omega\right)  $ by $-LX=\sum_{n=1}^{\infty}nI_{n}\left(
f_{n}\right)  $. The inverse of $L$ operating on centered r.v.'s in
$L^{2}\left(  \Omega\right)  $ is defined by the formula $-L^{-1}X=\sum
_{n=1}^{\infty}\frac{1}{n}I_{n}\left(  f_{n}\right)  .$
\end{definition}

The following formula will play an important role in our proofs where we use
Stein's lemma. It was originally noted in \cite{NP}. We provide a
self-contained proof of this result in the Appendix, which does not use the
concept of divergence operator (Skorohod integral).

\begin{lemma}
\label{lemkey}For any centered $X\in\mathbf{D}^{1,2}$ with a density and any
deterministic continuous and piecewise differentiable function $h$ such that
$h^{\prime}$ is bounded,%
\begin{equation}
\mathbf{E}\left[  Xh\left(  X\right)  \right]  =\mathbf{E}\left[  h^{\prime
}\left(  X\right)  \left\langle DX;-DL^{-1}X\right\rangle \right]  .
\label{easykey}%
\end{equation}

\end{lemma}

On the other hand, the next result and its proof (see \cite{NV}), make no
reference to Stein's lemma. Let the function $g$ be defined almost everywhere
by%
\begin{equation}
g\left(  z\right)  :=\mathbf{E}[\left\langle DX;-DL^{-1}X\right\rangle
|X=z].\label{gee}%
\end{equation}

\begin{proposition}
\label{nounouvivi}Let $X\in\mathbf{D}^{1,2}$ be centered with a density $\rho$
which is supported on a set $I$. Then $I$ is an interval $[a,b]$ and, with $g$
as above, we have for almost all $z\in(a,b)$,%
\[
\rho\left(  z\right)  =\frac{\mathbf{E}\left\vert Z\right\vert }{2g\left(
z\right)  }\exp\left(  -\int_{0}^{z}\frac{ydy}{g\left(  y\right)  }\right)  .
\]

\end{proposition}

Strictly speaking, the proof of this proposition is not contained in
\cite{NV}, since the authors there use the additional assumption that
$g\left(  x\right)  \geq1$ everywhere, which implies that $\rho$ exists\ and
that $I=\mathbf{R}$. However, the modification of their arguments to yield the
proposition above is straightforward, and we omit it: for instance, that $I$
is an interval follows from $X\in\mathbf{D}^{1,2}$ as seen in
\cite[Proposition 2.1.7]{Nbook}.\bigskip

As one can see from this proposition, and the statement of Theorem
\ref{thm12}, it is important to have a technique to be able to calculate
$DL^{-1}X$. We will use a device which can be found for instance in a
different form in the proof of Lemma 1.5.2 in \cite{Nbook}, and is at the core
of the so-called Mehler formula, also found in \cite{Nbook}. It requires a
special operator which introduces a coupling with an independent Wiener space.
This operator $R_{\theta}$ replaces $W$ by the linear combination $W\cos
\theta+W^{\prime}\sin\theta$ where $W^{\prime}$ is an independent copy of $W$.
For instance, if $W$ is Brownian motion and one writes the random variable $X$
as $X=F\left(  W\right)  $ where $F$ is a deterministic Borel-measurable
functional on the space of continuous functions, then
\begin{equation}
R_{\theta}X:=F\left(  W\cos\theta+W^{\prime}\sin\theta\right)  .\label{Rtheta}%
\end{equation}
We have the following formula (akin to the Mehler formula, and proved in the
Appendix), where $sgn\left(  \theta\right)  =\theta/\left\vert \theta
\right\vert $, where $\mathbf{E}^{\prime}$ represents the expectation w.r.t.
the randomness in $W^{\prime}$ only, i.e. conditional on $W$, and where
$D^{\prime}$ is the Malliavin derivative w.r.t. $W^{\prime}$ only.

\begin{lemma}
\label{lemDLX}For any $X\in\mathbf{D}^{1,2}$, for all $s\in T$,%
\[
-D_{s}\left(  L^{-1}X\right)  =\frac{1}{2}\int_{-\pi/2}^{\pi/2}\mathbf{E}%
^{\prime}\left[  D_{s}^{\prime}\left(  R_{\theta}X\right)  \right]
~sgn\left(  \theta\right)  d\theta.
\]

\end{lemma}

\section{Main results\label{MAIN}}

All results in this section are stated and discussed in the first two
subsections, the first one dealing with consequences of Stein's lemma, the
second with the function $g$. All proofs are in the third subsection.

\subsection{Results using Stein's lemma}

Our first result is tailored to Gaussian comparisons.

\begin{theorem}
\label{thmlb}Let $X\in\mathbf{D}^{1,2}$ be centered. Assume that almost
surely,
\begin{equation}
G:=\left\langle DX;-DL^{-1}X\right\rangle \geq1.\label{lbcond}%
\end{equation}
Then for every $z>0$,
\[
\mathbf{P}\left[  X>z\right]  \geq\bar{\Phi}\left(  z\right)  -\frac
{1}{1+z^{2}}\int_{z}^{\infty}2x\mathbf{P}\left[  X>x\right]  dx
\]

Assume instead that one has the reverse of inequality (\ref{lbcond}), then for
every $z>0$,
\[
\mathbf{P}\left[  X>z\right]  \leq\left(  1+\frac{1}{z^{2}}\right)  \bar{\Phi
}\left(  z\right)  .
\]

\end{theorem}

Before proving this theorem, we record some consequences of its lower bound
result in the next Corollary. In order to obtain a more precise lower bound
result on the tail $S\left(  z\right)  :=\mathbf{P}\left[  X>z\right]  $, it
appears to be necessary to make some regularity and integrability assumptions
on $S$. This is the aim of the second point in the next corollary. The first
and third points show what can be obtained by using only an integrability
condition, with no regularity assumption: we may either find a universal lower
bound on such quantities as $X$'s variance (the constant we find there may not
be of any special significance), or an asymptotic statement on $S$ itself.

\begin{corollary}
\label{thmlbcor}Let $X\in\mathbf{D}^{1,2}$ be centered. Let $S\left(
z\right)  :=\mathbf{P}\left[  X>z\right]  $. Assume that condition
(\ref{lbcond}) holds.

\begin{enumerate}
\item We have%
\[
Var\left[  X\right]  \geq K_{u}:=\frac{1}{\pi^{2}}\left(  2\sqrt{1+2\sqrt
{2\pi}}-1\right)  ^{2}\simeq0.21367.
\]

\item Assume there exists a constant $c>2$ such that $\left\vert S^{\prime
}\left(  z\right)  \right\vert /S\left(  z\right)  \leq c/z$ holds for large
$z$. Then for large $z$,%
\[
\mathbf{P}\left[  X>z\right]  \geq\frac{\left(  c-2\right)  \left(
1+z^{2}\right)  }{c-2+cz^{2}}\bar{\Phi}\left(  z\right)  \simeq\frac{\left(
c-2\right)  }{c}\bar{\Phi}\left(  z\right)  .
\]

\item Assume there exists a constant $c>2$ such that $S\left(  z\right)
<z^{-c}$ holds for large $z$. Then, for large $z$,%
\[
\sup_{x\geq z}x^{c}\mathbf{P}\left[  X>x\right]  \geq\frac{c-2}{c}z^{c}%
\bar{\Phi}\left(  z\right)  .
\]
Consequently,
\[
\limsup_{z\rightarrow\infty}\frac{\mathbf{P}\left[  X>z\right]  }{\bar{\Phi
}\left(  z\right)  }\geq\frac{c-2}{c}.
\]

\end{enumerate}
\end{corollary}

Let us discuss the assumptions and results in the corollary from a
quantitative standpoint. The assumption of point 2, $\left\vert S^{\prime
}\left(  z\right)  \right\vert /S\left(  z\right)  \leq c/z$, when integrated,
implies no more than existence of a moment of order larger than $2$; it does,
however, represent an additional monotonicity condition since it refers to
$S^{\prime}$. The assumption of point 3, which is weaker because it does not
require any monotonicity, also implies the same moment condition. This moment
condition is little more than the integrability required from $X$ belonging to
$\mathbf{D}^{1,2}$. If $c$ can be made arbitrarily large (for instance in
point $3$, this occurs when $X$ is assumed to have moments of all orders),
asymptotically $(c-2)/c$ can be replaced by $1$, yielding the sharpest
possible comparison to the normal tail. If indeed $S$ is close to the normal
tail, it is morally not a restriction to assume that $c$ can be taken
arbitrarily large: it is typically easy to check this via a priori estimates.

\subsection{Results using the function $g$}

We now present results which do not use Stein's lemma, but refer only to the
random variable $G:=\left\langle DX;-DL^{-1}X\right\rangle $ and the resulting
function $g\left(  z\right)  :=\mathbf{E}[G|X=z]$ introduced in (\ref{gee}).
We will prove the theorem below using the results in \cite{NV} on
representation of densities. Its corollary shows how to obtain quantitatively
explicit upper and lower bounds on the tail of a random variable, which are as
sharp as the upper and lower bounds one might establish on $g$. A description
of the advantages and disadvantages of using $g$ over Stein's lemma follows
the statements of the next theorem and its corollary.

\begin{theorem}
\label{ThmA}Let $X\in\mathbf{D}^{1,2}$ be centered. Let $G:=\left\langle
DX;-DL^{-1}X\right\rangle $ and $g\left(  z\right)  :=\mathbf{E}[G|X=z]$.
Assume that $X$ has a density which is positive on the interior of its support
$(a,+\infty)$. For $x\geq0$, let%
\[
A\left(  x\right)  :=\exp\left(  -\int_{0}^{x}\frac{ydy}{g\left(  y\right)
}\right)  .
\]
Then for all $x>0$,%
\begin{equation}
\mathbf{P}\left[  X>x\right]  =\frac{\mathbf{E}\left\vert X\right\vert }%
{2}\left(  \frac{A\left(  x\right)  }{x}-\int_{x}^{\infty}\frac{A\left(
y\right)  }{y^{2}}dy\right)  . \label{tailA}%
\end{equation}

\end{theorem}

\begin{remark}
The density formula in Proposition (\ref{nounouvivi}) shows that $g$ must be
non-negative. Assuming our centered $X\in\mathbf{D}^{1,2}$ has a density
$\rho$, we have already noted that $\rho$ must be positive on $(a,b)$ and zero
outside. To ensure that $b=+\infty$, as is needed in the above theorem, it is
sufficient to assume that $g$ is bounded below on $[0,b)$ by a positive
constant. If in addition we can assume, as in (\ref{lbcond}), that this
lower-boundedness of $g$ holds everywhere, then $X$ has a density, and its
support is $\mathbf{R}$.
\end{remark}

\begin{corollary}
\label{corA}Assume that for some $c^{\prime}\in\left(  0,1\right)  $ and some
$z_{0}>1$, we have for all $x>z_{0}$, $g\left(  x\right)  \leq c^{\prime}%
x^{2}$. Then, with $K:=\frac{\mathbf{E}\left\vert X\right\vert }{2}%
\frac{\left(  c^{\prime}\right)  ^{c^{\prime}}}{\left(  1+c^{\prime}\right)
^{1+c^{\prime}}},$ for $x>z_{0}$,%
\begin{equation}
\mathbf{P}\left[  X>x\right]  \geq K\frac{A\left(  x\right)  }{x}%
.\label{corAzero}%
\end{equation}

\begin{enumerate}
\item Under the additional assumption (\ref{lbcond}), $g\left(  x\right)
\geq1$ everywhere, and we have%
\[
\mathbf{P}\left[  X>z\right]  \geq K\frac{1}{x}\exp\left(  -\frac{x^{2}}%
{2}\right)  \simeq\sqrt{2\pi}K\bar{\Phi}\left(  z\right)  .
\]

\item If we have rather the stronger lower bound $g\left(  x\right)  \geq
c^{\prime\prime}x^{2}$ for some $c^{\prime\prime}\in(0,c^{\prime}]$ and all
$x>z_{0}$, then for $x>z_{0}$, and with some constant $K^{\prime}$ depending
on $g$, $c^{\prime\prime}$ and $z_{0}$,%
\[
\mathbf{P}\left[  X>z\right]  \geq K^{\prime}x^{-1-1/c^{\prime\prime}}.
\]

\item If we have instead that $g\left(  x\right)  \geq c_{1}x^{p}\ $for some
$c_{1}>0$, $p<2$, and for all $x>z_{0}$, then for $x>z_{0}$, and with some
constant $K^{\prime\prime}$ depending on $g$, $c_{1}$, $p$, and $z_{0}$,%
\[
\mathbf{P}\left[  X>z\right]  \geq K^{\prime\prime}\exp\left(  -\frac{x^{2-p}%
}{\left(  2-p\right)  c_{1}}\right)  .
\]

\item In the last two points, if the inequalities on $g$ in the hypotheses are
reversed, the conclusions are also reversed, without changing any of the
constants.\bigskip
\end{enumerate}
\end{corollary}

The tail formula (\ref{tailA}) in Theorem \ref{ThmA} readily implies
asymptotic estimates on $S$ of non-Gaussian type if one is able to compare $g$
to a power function. Methods using Stein's lemma, at least in its form
described in Section \ref{STEIN}, only work efficiently for comparing $S$ to
the Gaussian tail. Arguments found in Nourdin and Peccati's articles (e.g.
\cite{NP}) indicate that Stein's method may be of use in some specific
non-Gaussian cases, which one could use to compare tails to the Gamma tail,
and perhaps to other tails in the Pearson family, which would correspond to
polynomial $g$ with degree at most $2$. The flexibility of our method of
working directly with $g$ rather than Stein's lemma, is that it seems to allow
any type of tail. Stein's method has one important advantage, however: it is
not restricted to having a good control on $g$; Theorem \ref{thmlb}
establishes Gaussian lower bounds on tails by only assuming (\ref{lbcond}) and
mild conditions on the tail itself. This is to be compared to the lower bound
\cite[Theorem 4.2]{NV} proved via the function $g$ alone, where it required
growth conditions on $g$ which may not be that easy to check.

There is one intriguing, albeit perhaps technical, fact regarding the use of
Stein's method: in Point 1 of the above Corollary \ref{corA}, since the
comparison is made with a Gaussian tail, one may wonder what the usage Stein's
lemma via Theorem \ref{thmlb} may produce when assuming, as in Point 1 of
Corollary \ref{corA}, that $g\left(  x\right)  \geq1$ and $g$ grows slower
than $x^{2}$. As it turns out, Stein's method is not systematically superior
to Corollary \ref{corA}, as we now see.

\begin{corollary}
[Consequence of Theorem \ref{thmlb}]\label{corlblast}Assume that $g\left(
x\right)  \geq1$ and, for some $c^{\prime}<1$ and large $x>z_{0}$, $g\left(
x\right)  \leq c^{\prime}x^{2}$. Then for $z>z_{0}$,
\[
\mathbf{P}\left[  X>z\right]  \geq\frac{1+z^{2}}{1+\left(  2c^{\prime
}+1\right)  z^{2}}\bar{\Phi}\left(  z\right)  \simeq\frac{1}{2c^{\prime}%
+1}\bar{\Phi}\left(  z\right)  .
\]

\end{corollary}

When this corollary and Point 1 in Corollary \ref{corA} are used in an
efficient situation, this means that $X$ is presumably \textquotedblleft
subgaussian\textquotedblright\ as well as being \textquotedblleft
supergaussian\textquotedblright\ as a consequence of assumption (\ref{lbcond}%
). For illustrative purposes, we can translate this roughly as meaning that
$g\left(  x\right)  $ is in the interval, say, $[1,1+\varepsilon]$ for all
$x$. This implies that we can take $c^{\prime}\rightarrow0$ in both
Corollaries \ref{corA} and \ref{corlblast}; as a consequence, the first
corollary yields $\mathbf{P}\left[  X>z\right]  \geq\bar{\Phi}\left(
z\right)  $, while the second gives $\mathbf{P}\left[  X>z\right]  \geq
(\sqrt{2\pi}\mathbf{E}\left\vert X\right\vert /2)~\bar{\Phi}\left(  z\right)
$. The superiority of one method over another then depends on how $\sqrt{2\pi
}\mathbf{E}\left\vert X\right\vert /2$ compares to $1$. It is elementary to
check that, in \textquotedblleft very sharp\textquotedblright\ situations,
which means that $\varepsilon$ is quite small, $\sqrt{2\pi}\mathbf{E}%
\left\vert X\right\vert /2$ will be close to $1$, from which one can only
conclude that both methods appear to be equally efficient.

\subsection{Proofs}

We now turn to the proofs of the above results.\vspace{0.05in}

\begin{proof}
[Proof of Theorem \ref{thmlb}]\emph{Step 1: exploiting the negativity of}
$f^{\prime}$. From lemma \ref{Steinlem}, we are able to calculate the
derivative of the solution $f$ to Stein's equation:

\begin{itemize}
\item for $x\leq z$, $f^{\prime}\left(  x\right)  =\bar{\Phi}\left(  z\right)
\left(  1+\sqrt{2\pi}\left(  1-\bar{\Phi}\left(  x\right)  \right)
xe^{x^{2}/2}\right)  ;$

\item for $x>z$, $f^{\prime}\left(  x\right)  =\left(  1-\bar{\Phi}\left(
z\right)  \right)  \left(  -1+\sqrt{2\pi}\bar{\Phi}\left(  x\right)
xe^{x^{2}/2}\right)  .$
\end{itemize}

We now use the standard estimate, valid for all $x>0$:%
\begin{equation}
\frac{x}{(x^{2}+1)\sqrt{2\pi}}e^{-x^{2}/2}\leq\bar{\Phi}\left(  x\right)
\leq\frac{1}{x\sqrt{2\pi}}e^{-x^{2}/2}. \label{vieuxtruc}%
\end{equation}
In the case $x>z$, since $z>0$, the upper estimate yields $f^{\prime}\left(
x\right)  \leq\left(  1-\bar{\Phi}\left(  z\right)  \right)  \left(
-1+1\right)  =0$. Now by the expression for $\mathbf{P}\left[  X>z\right]  $
in Corollary \ref{Steincor}, the negativity of $f^{\prime}$ on $\left\{
x>z\right\}  $ implies for all $z>0$,%
\begin{align*}
\mathbf{P}\left[  X>z\right]   &  =\bar{\Phi}\left(  z\right)  -\mathbf{E}%
\left[  \mathbf{1}_{X\leq z}f^{\prime}\left(  X\right)  \right]
-\mathbf{E}\left[  \mathbf{1}_{X>z}f^{\prime}\left(  X\right)  \right]
+\mathbf{E}\left[  Xf\left(  X\right)  \right] \\
&  \geq\bar{\Phi}\left(  z\right)  -\mathbf{E}\left[  \mathbf{1}_{X\leq
z}f^{\prime}\left(  X\right)  \right]  +\mathbf{E}\left[  Xf\left(  X\right)
\right]  .
\end{align*}
\vspace{0.2cm}

\emph{Step 2: Exploiting the positivities and the smallness of }$f^{\prime}$.
Using Step 1, we have%
\[
\mathbf{P}\left[  X>z\right]  \geq\bar{\Phi}\left(  z\right)  -\mathbf{E}%
\left[  \mathbf{1}_{X\leq z}f^{\prime}(X)\right]  +\mathbf{E}\left[
\mathbf{1}_{X\leq z}Xf(X)\right]  +\mathbf{E}\left[  \mathbf{1}_{X>z}%
Xf(X)\right]  .
\]
We apply Lemma \ref{lemkey} to the function $h\left(  x\right)  =\left(
f\left(  x\right)  -f\left(  z\right)  \right)  \mathbf{1}_{x\leq z}$; $h$ is
continuous everywhere; it is differentiable everywhere with a bounded
derivative, equal to $f^{\prime}\left(  x\right)  \mathbf{1}_{x\leq z}$,
except at $x=z$. Thus we get%
\begin{align}
\mathbf{P}\left[  X>z\right]   &  \geq\bar{\Phi}\left(  z\right)
-\mathbf{E}\left[  \mathbf{1}_{X\leq z}f^{\prime}(X)\right]  +\mathbf{E}%
\left[  Xh(X)\right]  +\mathbf{E}\left[  \mathbf{1}_{X\leq z}X\right]
f\left(  z\right)  +\mathbf{E}\left[  \mathbf{1}_{X>z}Xf(X)\right] \nonumber\\
&  \geq\bar{\Phi}\left(  z\right)  +\mathbf{E}\left[  \mathbf{1}_{X\leq
z}f^{\prime}\left(  X\right)  \left(  -1+\left\langle DX;-DL^{-1}%
X\right\rangle \right)  \right]  +\mathbf{E}\left[  \mathbf{1}_{X\leq
z}X\right]  f\left(  z\right)  +\mathbf{E}\left[  \mathbf{1}_{X>z}%
Xf(X)\right]  . \label{finaleq}%
\end{align}

When $x\leq z$, we can use the formula in Step 1 to prove that $f^{\prime
}\left(  x\right)  \geq0$. Indeed this is trivial when $x\geq0$, while when
$x<0$, it is proved as follows: for $x=-y<0$, and using the upper bound in
(\ref{vieuxtruc})%
\[
f^{\prime}\left(  x\right)  =\bar{\Phi}\left(  z\right)  \left(  1+\sqrt{2\pi
}\left(  1-\bar{\Phi}\left(  x\right)  \right)  xe^{x^{2}/2}\right)
=\bar{\Phi}\left(  z\right)  \left(  1-\sqrt{2\pi}\bar{\Phi}\left(  y\right)
ye^{y^{2}/2}\right)  \geq0.
\]
By the lower bound hypothesis (\ref{lbcond}), we also have positivity of
$-1+\left\langle DX;-DL^{-1}X\right\rangle $. Thus the second term on the
right-hand side of (\ref{finaleq}) is non-negative. In other words we have%
\begin{align}
\mathbf{P}\left[  X>z\right]   &  \geq\bar{\Phi}\left(  z\right)
+\mathbf{E}\left[  \mathbf{1}_{X\leq z}X\right]  f\left(  z\right)
+\mathbf{E}\left[  \mathbf{1}_{X>z}Xf(X)\right] \label{finaleq3}\\
&  =:\bar{\Phi}\left(  z\right)  +A \label{finaleq2}%
\end{align}

The sum of the last two terms on the right-hand side of (\ref{finaleq3}),
which we call $A$, can be rewritten as follows, using the fact that
$\mathbf{E}\left[  X\right]  =0$:%
\begin{align*}
A  &  :=\mathbf{E}\left[  \mathbf{1}_{X\leq z}X\right]  f\left(  z\right)
+\mathbf{E}\left[  \mathbf{1}_{X>z}Xf(X)\right] \\
&  =\mathbf{E}\left[  \mathbf{1}_{X\leq z}X\right]  f\left(  z\right)
+\mathbf{E}\left[  \mathbf{1}_{X>z}X\left(  f(X)-f\left(  z\right)  \right)
\right]  +f\left(  z\right)  \mathbf{E}\left[  \mathbf{1}_{X>z}X\right] \\
&  =\mathbf{E}\left[  \mathbf{1}_{X>z}X\left(  f(X)-f\left(  z\right)
\right)  \right]  .
\end{align*}
This quantity $A$ is slightly problematic since, $f$ being decreasing on
$[z,+\infty)$, we have $A<0$. However, we can write $f(X)-f\left(  z\right)
=f^{\prime}\left(  \xi\right)  \left(  X-z\right)  $ for some random $\xi>z$.
Next we use the lower bound in (\ref{vieuxtruc}) to get that for all $\xi>z$,
\begin{equation}
\left\vert f^{\prime}\left(  \xi\right)  \right\vert =-f^{\prime}\left(
\xi\right)  =\left(  1-\bar{\Phi}\left(  z\right)  \right)  \left(
1-\sqrt{2\pi}\bar{\Phi}\left(  \xi\right)  \xi e^{\xi^{2}/2}\right)
\leq1\cdot\left(  1-\frac{\xi^{2}}{1+\xi^{2}}\right)  =\frac{1}{1+\xi^{2}}.
\label{tinyf'}%
\end{equation}
This upper bound can obviously be further bounded above uniformly by $\left(
1+z^{2}\right)  ^{-1}$, which means that
\[
\left\vert A\right\vert \leq\mathbf{E}\left[  \mathbf{1}_{X>z}X\left(
X-z\right)  \right]  \frac{1}{1+z^{2}}\leq\mathbf{E}\left[  \mathbf{1}%
_{X>z}X^{2}\right]  \frac{1}{1+z^{2}}.
\]
By using this estimate in (\ref{finaleq2}) we finally get%
\begin{equation}
\mathbf{P}\left[  X>z\right]  \geq\bar{\Phi}\left(  z\right)  -\mathbf{E}%
\left[  \mathbf{1}_{X>z}X^{2}\right]  \frac{1}{1+z^{2}}. \label{finaleq4}%
\end{equation}

\emph{Step 3: integrating by parts.} For notational compactness, let $S\left(
z\right)  :=\mathbf{P}\left[  X>z\right]  $. We integrate the last term in
(\ref{finaleq4}) by parts with respect to the positive measure $-dS\left(
x\right)  $. We have, for any $z>0$,%
\begin{align*}
\mathbf{E}\left[  \mathbf{1}_{X>z}X^{2}\right]   &  =-\int_{z}^{\infty}%
x^{2}dS\left(  x\right)  =z^{2}S\left(  x\right)  -\lim_{x\rightarrow+\infty
}x^{2}S\left(  x\right)  +\int_{z}^{\infty}2xS\left(  x\right)  dx\\
&  \leq z^{2}S\left(  z\right)  +\int_{z}^{\infty}2xS\left(  x\right)  dx.
\end{align*}
The conclusion (\ref{finaleq4}) from the previous step now implies%
\[
S\left(  z\right)  \geq\bar{\Phi}\left(  z\right)  -\frac{z^{2}}{1+z^{2}%
}S\left(  z\right)  -\frac{1}{1+z^{2}}\int_{z}^{\infty}2xS\left(  x\right)
dx,
\]
which finishes the proof of the theorem's lower bound.\vspace{0.2cm}

\emph{Step 4: Upper bound}. The proof of the upper bound is similar to, not
symmetric with, and less delicate than, the proof of the lower bound. Indeed,
we can take advantage of a projective positivity result on the inner product
of $DX$ and $-DL^{-1}X$, namely \cite[Proposition 3.9]{NP} which says that
$\mathbf{E}\left[  \left\langle DX;-DL^{-1}X\right\rangle |X\right]  \geq0$.
This allows us to avoid the need for any additional moment assumptions. Using
Lemma \ref{lemkey} directly with the function $h=f$, which is continuous, and
differentiable everywhere except at $x=z$, we have%
\begin{align}
&  \mathbf{P}\left[  X>z\right] \nonumber\\
&  =\bar{\Phi}\left(  z\right)  -\mathbf{E}\left[  f^{\prime}\left(  X\right)
\right]  +\mathbf{E}\left[  f^{\prime}\left(  X\right)  \left\langle
DX;-DL^{-1}X\right\rangle \right] \nonumber\\
&  =\bar{\Phi}\left(  z\right)  +\mathbf{E}\left[  \mathbf{1}_{X\leq
z}f^{\prime}\left(  X\right)  \left(  -1+\left\langle DX;-DL^{-1}%
X\right\rangle \right)  \right]  +\mathbf{E}\left[  \mathbf{1}_{X>z}f^{\prime
}\left(  X\right)  \left(  -1+\left\langle DX;-DL^{-1}X\right\rangle \right)
\right] \nonumber\\
&  \leq\bar{\Phi}\left(  z\right)  +\mathbf{E}\left[  \mathbf{1}%
_{X>z}f^{\prime}\left(  X\right)  \left(  -1+\left\langle DX;-DL^{-1}%
X\right\rangle \right)  \right]  \label{apositiver}%
\end{align}
where the last inequality simply comes from the facts that by hypothesis
$-1+\left\langle DX;-DL^{-1}X\right\rangle $ is negative, and when $x\leq z$,
$f^{\prime}\left(  x\right)  \geq0$ (see previous step for proof of this
positivity). It remains to control the term in (\ref{apositiver}): since
$\mathbf{E}\left[  \left\langle DX;-DL^{-1}X\right\rangle |X\right]  \geq0$,
and using the negativity of $f^{\prime}$ on $x>z$,%
\begin{align*}
\mathbf{E}\left[  \mathbf{1}_{X>z}f^{\prime}\left(  X\right)  \left(
-1+\left\langle DX;-DL^{-1}X\right\rangle \right)  \right]   &  =\mathbf{E}%
\left[  \mathbf{1}_{X>z}f^{\prime}\left(  X\right)  \mathbf{E}\left[  \left(
-1+\left\langle DX;-DL^{-1}X\right\rangle \right)  |X\right]  \right] \\
&  \leq-\mathbf{E}\left[  \mathbf{1}_{X>z}f^{\prime}\left(  X\right)  \right]
=\mathbf{E}\left[  \mathbf{1}_{X>z}\left\vert f^{\prime}\left(  X\right)
\right\vert \right]
\end{align*}
This last inequality together with the bound on $f^{\prime}$ obtained in
(\ref{tinyf'}) imply%
\[
\mathbf{E}\left[  \mathbf{1}_{X>z}f^{\prime}\left(  X\right)  \left(
-1+\left\langle DX;-DL^{-1}X\right\rangle \right)  \right]  \leq
\mathbf{P}\left[  X>z\right]  \frac{1}{1+z^{2}}.
\]
Thus we have proved that%
\[
\mathbf{P}\left[  X>z\right]  \leq\bar{\Phi}\left(  z\right)  +\mathbf{P}%
\left[  X>z\right]  \frac{1}{1+z^{2}}%
\]
which implies the upper bound of the theorem, finishing its proof.$\vspace
{0.3cm}$
\end{proof}

\begin{proof}
[Proof of Corollary \ref{thmlbcor}]\emph{Proof of Point 2.} One notes first
that by a result in \cite{NV}, condition (\ref{lbcond}) implies that $X$ has a
density, so that $S^{\prime}$ is defined. Then we get%
\begin{align*}
F\left(  z\right)   &  :=\int_{z}^{\infty}x\mathbf{P}\left[  X>z\right]
dx=\int_{z}^{\infty}xS\left(  x\right)  dx\leq\frac{1}{c}\int_{z}^{\infty
}x^{2}\left\vert S^{\prime}\left(  x\right)  \right\vert dx\\
&  =\frac{1}{c}\left(  z^{2}S\left(  z\right)  -\lim_{x\rightarrow\infty}%
x^{2}S\left(  x\right)  +\int_{z}^{\infty}2xS\left(  x\right)  dx\right) \\
&  \leq\frac{1}{c}\left(  z^{2}S\left(  z\right)  +2F\left(  x\right)
\right)
\end{align*}
which implies%
\[
F\left(  z\right)  \leq\frac{1}{c-2}z^{2}S\left(  z\right)  .
\]
With the lower bound conclusion of Theorem \ref{thmlb}, we obtain%
\[
S\left(  z\right)  \geq\bar{\Phi}\left(  z\right)  -\frac{2z^{2}}{1+z^{2}%
}\frac{1}{c-2}S\left(  z\right)
\]
which is equivalent to the statement of Point 2.\vspace{0.1in}

\emph{Proof of Point 3.} From Theorem \ref{thmlb}, we have for large $z$,%
\begin{align*}
S\left(  z\right)   &  \geq\bar{\Phi}\left(  z\right)  -\frac{1}{1+z^{2}}%
\int_{z}^{\infty}2x^{1-c}x^{c}S\left(  x\right)  dx\\
&  \geq\bar{\Phi}\left(  z\right)  -\frac{1}{1+z^{2}}\sup_{x>z}\left[
x^{c}S\left(  x\right)  \right]  \int_{z}^{\infty}2x^{1-c}dx=\bar{\Phi}\left(
z\right)  -\frac{z^{2-c}2/\left(  c-2\right)  }{1+z^{2}}\sup_{x>z}\left[
x^{c}S\left(  x\right)  \right]
\end{align*}
which implies%
\[
\sup_{x>z}\left[  x^{c}S\left(  x\right)  \right]  \left(  \frac{2}%
{c-2}+1\right)  \geq z^{c}\bar{\Phi}\left(  z\right)
\]
which is equivalent to the first part of the statement of Point 3, the second
part following from the fact that $z^{c}\bar{\Phi}\left(  z\right)  $ is
decreasing for large $z$.\vspace{0.1in}

\emph{Proof of Point 1.} As in Point 2, we define $F\left(  z\right)
:=\int_{z}^{\infty}x\mathbf{P}\left[  X>z\right]  dx$ but this time, we do not
need to use the density of $X$. Instead, we note that by integration by
parts,
\[
\mathbf{E}\left[  X^{2}\mathbf{1}_{X>z}\right]  =z^{2}S\left(  z\right)
-\lim_{x\rightarrow\infty}x^{2}S\left(  x\right)  +2F\left(  z\right)  ,
\]
Since $X\in L^{2}\left(  \Omega\right)  $, $\lim_{x\rightarrow\infty}%
x^{2}S\left(  x\right)  \leq\lim_{x\rightarrow\infty}\mathbf{E}\left[
X^{2}\mathbf{1}_{X>x}\right]  =0$ and therefore $\mathbf{E}\left[
X^{2}\mathbf{1}_{X>0}\right]  =2F\left(  0\right)  .$ Since our hypothesis is
invariant with respect to changing $X$ into $-X$, we also get $\mathbf{E}%
\left[  X^{2}\mathbf{1}_{X<0}\right]  =2F\left(  0\right)  .$ Therefore%
\[
Var\left[  X\right]  =4F\left(  0\right)  .
\]
Thus we only need to find a lower bound on $F\left(  0\right)  $.

Now let $p=1,2$, and define%
\[
a_{p}:=\mathbf{E}\left[  X^{p}\mathbf{1}_{X>0}\right]  .
\]
Thus we have $a_{2}=F\left(  0\right)  $ and $a_{1}=\mathbf{E}\left[
X_{+}\right]  $. Now, by integration by parts, $a_{1}\geq\int_{0}^{\infty
}S\left(  x\right)  dx$. Using Theorem \ref{thmlb}, we thus get%
\begin{align*}
a_{1}  &  \geq\int_{0}^{\infty}\bar{\Phi}\left(  x\right)  dx-2\int
_{0}^{\infty}\frac{1}{1+x^{2}}F\left(  x\right)  dx\geq\int_{0}^{\infty}%
\bar{\Phi}\left(  x\right)  dx-2F\left(  0\right)  \int_{0}^{\infty}\frac
{1}{1+x^{2}}dx\\
&  =\frac{1}{\sqrt{2\pi}}-\pi a_{2}.
\end{align*}
Since $a_{2}\geq a_{1}^{2}$, this proves that $F\left(  0\right)  =a_{2}%
\geq\left(  \sqrt{1+2\sqrt{2\pi}}-1\right)  ^{2}/\left(  2\pi\right)  ^{2}$
and the conclusion follows.\vspace*{0.1in}
\end{proof}

\begin{proof}
[Proof of Theorem \ref{ThmA}]By Proposition \ref{nounouvivi}, with
$L=\mathbf{E}\left\vert X\right\vert /2$, for $x\in(a,b)$,
\[
\rho\left(  x\right)  =LA\left(  x\right)  /g\left(  x\right)  .
\]
By definition we also get $A^{\prime}\left(  x\right)  =-xA\left(  x\right)
/g\left(  x\right)  =-xL^{-1}\rho\left(  x\right)  $, and thus
\[
\mathbf{P}\left[  X>x\right]  =:S\left(  x\right)  =L\int_{x}^{+\infty}%
\frac{-A^{\prime}}{y}dy=L\left(  \frac{A\left(  x\right)  }{x}-\lim
_{y\rightarrow\infty}\frac{A\left(  y\right)  }{y}-\int_{x}^{+\infty}%
\frac{A\left(  y\right)  }{y^{2}}dy\right)  .
\]
Since $g$ is non-negative, $A$ is bounded, and the term $\lim_{y\rightarrow
\infty}A\left(  y\right)  /y$ is thus zero. Equality (\ref{tailA}) follows
immediately, proving the theorem.\vspace*{0.1in}
\end{proof}

\begin{proof}
[Proof of Corollary \ref{corA}]\emph{Proof of inequality (\ref{corAzero})}.
From Theorem \ref{ThmA}, with $L=\mathbf{E}\left\vert X\right\vert /2$, and
$k>1$, and using the fact that $A$ is decreasing, we can write%
\begin{align*}
S\left(  x\right)   &  =:\mathbf{P}\left[  X>x\right]  =L\left(
\frac{A\left(  x\right)  }{x}-\int_{x}^{kx}\frac{A\left(  y\right)  }{y^{2}%
}dy-\int_{kx}^{+\infty}\frac{A\left(  y\right)  }{y^{2}}dy\right) \\
&  \geq L\left(  \frac{A\left(  x\right)  }{x}-\frac{A\left(  x\right)  }%
{x}\left(  1-\frac{1}{k}\right)  -\frac{A\left(  kx\right)  }{kx}\right) \\
&  =L\frac{A\left(  x\right)  }{x}\frac{1}{k}\left(  1-\frac{A\left(
kx\right)  }{A\left(  x\right)  }\right)  .
\end{align*}
It is now just a matter of using the assumption $g\left(  x\right)  \leq
c^{\prime}x^{2}$ to control $A\left(  kx\right)  /A\left(  x\right)  $. We
have for large $x$,%
\[
\frac{A\left(  kx\right)  }{A\left(  x\right)  }=\exp\left(  -\int_{x}%
^{kx}\frac{ydy}{g\left(  y\right)  }\right)  \leq\exp\left(  -\frac
{1}{c^{\prime}}\log k\right)  =k^{-1/c^{\prime}}.
\]
\vspace*{0.1in}This proves%
\[
S\left(  x\right)  \geq L\frac{A\left(  x\right)  }{x}\frac{1}{k}\left(
1-k^{-1/c^{\prime}}\right)  .
\]
The proof is completed simply by optimizing this over the values of $k>1$: the
function $k\mapsto\left(  1-k^{-1/c^{\prime}}\right)  /k$ reaches its maximum
of $\left(  c^{\prime}\right)  ^{c^{\prime}}\left(  1+c^{\prime}\right)
^{-c^{\prime}-1}$ at $\left(  1+1/c^{\prime}\right)  ^{c^{\prime}}$%
.\vspace*{0.1in}

\emph{Proof of Points 1, 2, 3, and 4.} Point 1 is immediate since
\thinspace$g\left(  x\right)  \geq1$ implies $A\left(  x\right)  \geq
\exp\left(  -x^{2}/2\right)  $. Similarly, for Point 2, we have%
\[
A\left(  x\right)  \geq\exp\left(  -\int_{0}^{y_{0}}\frac{ydy}{g\left(
y\right)  }\right)  \exp\left(  -\frac{1}{c^{\prime\prime}}\int_{y_{0}}%
^{x}\frac{dy}{y}\right)  =cst~x^{-1/c^{\prime\prime}},
\]
and Point 3 follows in the same fashion. Point 4 is shown identically by
reversing all inequalities, conclusing the proof of the
Corollary.\emph{\vspace*{0.1in}}
\end{proof}

\begin{proof}
[Proof of Corollary \ref{corlblast}]This is in fact a corollary of the proof
of Theorem \ref{thmlb}. At the end of Step 2 therein, in (\ref{finaleq4}), we
prove that (\ref{lbcond}), the lower bound assumption $\left\langle
DX;-DL^{-1}X\right\rangle =:G\geq1$, implies%
\begin{equation}
S\left(  z\right)  \geq\bar{\Phi}\left(  z\right)  -\mathbf{E}\left[
\mathbf{1}_{X>z}X\left(  X-z\right)  \right]  \frac{1}{1+z^{2}}%
.\label{4proofpoint1}%
\end{equation}
Let us investigate the term $B:=\mathbf{E}\left[  \mathbf{1}_{X>z}X\left(
X-z\right)  \right]  $. Using Lemma \ref{lemkey} with the function $h\left(
x\right)  =\left(  x-z\right)  \mathbf{1}_{x>z}$, we have
\[
B=\mathbf{E}\left[  \mathbf{1}_{X>z}G\right]  .
\]
Now use the upper bound assumption on $G$: we get, for all $z\geq z_{0}$,%
\begin{align}
B &  \leq c^{\prime}\mathbf{E}\left[  \mathbf{1}_{X>z}X^{2}\right]
=c^{\prime}\mathbf{E}\left[  \mathbf{1}_{X>z}X\left(  X-z\right)  \right]
+c^{\prime}z\mathbf{E}\left[  \mathbf{1}_{X>z}X\right]  \nonumber\\
&  =c^{\prime}B+c^{\prime}z\mathbf{E}\left[  \mathbf{1}_{X>z}X\right]
=c^{\prime}B+c^{\prime}z\left(  zS\left(  z\right)  +\int_{z}^{\infty}S\left(
x\right)  dx\right)  ,\label{Bub}%
\end{align}
where we used integration by parts for the last inequality. Integration by
parts also directly shows:%
\[
B=2\int_{z}^{\infty}xS\left(  s\right)  -z\int_{z}^{\infty}S\left(  x\right)
dx.
\]
Introducing the following additional notation: $D:=z\int_{z}^{\infty}S\left(
x\right)  dx$ and $E:=2\int_{z}^{\infty}xS\left(  s\right)  $, we see that
$B=E-D$ and also that $E\geq2D$. Moreover, in (\ref{Bub}), we also recognize
the appearance of $D$. Therefore we have%
\[
\left(  E-D\right)  \left(  1-c^{\prime}\right)  \leq c^{\prime}D+c^{\prime
}z^{2}S\left(  z\right)  \leq\left(  c^{\prime}/2\right)  E+c^{\prime}%
z^{2}S\left(  z\right)
\]
which easily implies%
\[
B\leq E\leq2c^{\prime}z^{2}S\left(  z\right)  .
\]
From (\ref{4proofpoint1}), we now get%
\[
S\left(  z\right)  \geq\bar{\Phi}\left(  z\right)  -\frac{2c^{\prime}z^{2}%
}{1+z^{2}}S\left(  z\right)
\]
from which we obtain, for $z\geq z_{0}$%
\[
S\left(  z\right)  \geq\frac{1+z^{2}}{1+\left(  2c^{\prime}+1\right)  z^{2}%
}\bar{\Phi}\left(  z\right)  ,
\]
finishing the proof of the corollary.\vspace*{0.1in}
\end{proof}

\section{Fluctuation exponent and deviations for polymers in Gaussian
environments\label{FLUCTU}}

Lemma \ref{lemDLX} provides a way to calculate $\left\langle DX;-DL^{-1}%
X\right\rangle $ in order to check, for instance, whether it is bounded below
by a positive constant $c^{2}$. If $c^{2}\neq1$, because of the bilinearity of
Condition (\ref{lbcond}), one only needs to consider $X/c$ instead of $X$ in
order to apply Theorem \ref{thmlb}, say. To show that such a tool can be
applied with ease in a non-trivial situation, we have chosen the issue of
fluctuation exponents for polymers in random environments.

We can consider various polymer models in random environments constructed by
analogy with the so-called stochastic Anderson models (see \cite{RT} and
\cite{FV}). A polymer's state space $R$ can be either $\mathbf{R}^{d}$ or the
$d$-dimensional torus $\mathbf{S}^{d}$, or also $\mathbf{Z}^{d}$ or
$\mathbf{Z}/p\mathbf{Z}$; we could also use any Lie group for $R$. We can
equip $R$ with a Markov process $b$ on $[0,\infty)$ whose infinitesimal
generator, under the probability measure $P_{b}$, is the Laplace(-Beltrami)
operator or the discrete Laplacian. Thus for instance, $b$ is Brownian motion
when $R=\mathbf{R}^{d}$, or is the simple symmetric random walk when
$R=\mathbf{Z}^{d}$; it is the image of Brownian motion by the imaginary
exponential map when $R=\mathbf{S}^{1}$. To simplify our exposition, we can
and will typically assume, unless explicitly stated otherwise, that
$R=\mathbf{R}$, but our constructions and proofs can be adapted to any of the
above choices.

\subsection{The random environment\label{RE}}

Let $W$ be a Gaussian field on $\mathbf{R}_{+}\times\mathbf{R}$ which is
homogeneous in space and is Brownian in time for fixed space parameter: the
covariance of $W$ is thus%
\[
\mathbf{E}\left[  W\left(  t,x\right)  W\left(  s,y\right)  \right]
=\min\left(  s,t\right)  Q\left(  x-y\right)  ,
\]
for some homogeneous covariance function $Q$ on $\mathbf{R}$. We assume that
$Q$ is continuous and that its Fourier transform $\hat{Q}$ is a measure with a
density also denoted by $\hat{Q}$. Note that $\hat{Q}$ is positive and
$\left\vert Q\right\vert $ is bounded by $Q\left(  0\right)  $. The field $W$
can be represented using a very specific isonormal Gaussian process: there
exists a white noise measure $M$ on $\mathbf{R}_{+}\times\mathbf{R}$ such that%
\[
W\left(  t,x\right)  =\int_{0}^{t}\int_{\mathbf{R}}M\left(  ds,d\lambda
\right)  ~\sqrt{\hat{Q}\left(  \lambda\right)  }~e^{i\lambda\cdot x},
\]
where the above integral is the Wiener integral of $(s,\lambda)\mapsto
\mathbf{1}_{[0,t]}\left(  s\right)  \sqrt{\hat{Q}\left(  \lambda\right)
}~e^{i\lambda\cdot x}$ with respect to $M$. This $M$ is an isonormal Gaussian
process whose Hilbert space is $H=L^{2}(\mathbf{R}_{+}\times\mathbf{R})$.
Malliavin derivatives relative to $M$ will take their parameters $\left(
s,\lambda\right)  $ in $\mathbf{R}_{+}\times\mathbf{R}$, and inner products
and norms are understood in $H$. There is a slight possibility of notational
confusion since now the underlying isonormal Gaussian process is called $M$,
with the letter $W$ -- the traditional name of the polymer potential field --
being a linear transformation of $M$.

The relation between $D$ and $W$ is thus that $D_{s,\lambda}W\left(
t,x\right)  =e^{i\lambda\cdot x}\sqrt{\hat{Q}\left(  \lambda\right)
}\mathbf{1}_{[0,t]}\left(  s\right)  $. We will make use of the following
similarly important formulas: for any measurable function $f$:
\begin{align}
D_{s,\lambda}\int_{\mathbf{R}}\int_{0}^{t}M\left(  ds,d\lambda\right)
\sqrt{\hat{Q}\left(  \lambda\right)  }e^{i\lambda\cdot f\left(  s\right)  }
&  =\sqrt{\hat{Q}\left(  \lambda\right)  }e^{i\lambda\cdot f\left(  s\right)
};\label{DM}\\
\int_{0}^{t}\int_{\mathbf{R}}ds~\hat{Q}\left(  \lambda\right)  d\lambda
~e^{i\lambda\cdot f\left(  s\right)  }  &  =\int_{0}^{t}Q\left(  f\left(
s\right)  \right)  ds \label{QhatQ}%
\end{align}
Quantitatively, this calculation will be particularly useful as a key to easy
upper bounds by noting the fact that $\max_{x\in\mathbf{R}}Q\left(  x\right)
=Q\left(  0\right)  $ is positive and finite. On the other hand, if $Q$ is
positive and non-degenerate, lower bounds will easily follow.

In order to use the full strength of our estimates in Section (\ref{MAIN}), we
will also allow $Q$ to be inhomogeneous, and in particular, unbounded. This is
easily modeled by specifying that
\[
W\left(  t,x\right)  =\int_{0}^{t}\int_{\mathbf{R}}M\left(  ds,d\lambda
\right)  ~q\left(  \lambda,x\right)
\]
where $\int_{\mathbf{R}}q\left(  \lambda,x\right)  q\left(  \lambda,y\right)
=Q\left(  x,y\right)  $. Calculations similar to (\ref{DM}) and (\ref{QhatQ})
then ensue.

We may also devise polymer models in non-Gaussian environments by considering
$W$ as a mixture of Gaussian fields. This means that we consider $Q$ to be
random itself, with respect to some separate probability space. We will place
only weak restrictions on this randomness: under a probability measure
$\mathcal{P}$, we assume $\hat{Q}$ is a non-negative random field on
$\mathbf{R}$, integrable on $\mathbf{R}$, with $Q\left(  0\right)
=\int_{\mathbf{R}}\hat{Q}\left(  \lambda\right)  d\lambda$ integrable with
respect to $\mathcal{P}$.

\subsection{The polymer and its fluctuation exponent}

Let the Hamiltonian of a path $b$ in $\mathbf{R}$ under the random environment
$W$ be defined, up to time $t$, as%
\[
H_{t}^{W}\left(  b\right)  =\int_{0}^{t}W\left(  ds,b_{s}\right)
=\int_{\mathbf{R}}\int_{0}^{t}M\left(  ds,d\lambda\right)  e^{i\lambda\cdot
b_{s}}.
\]
Since $W$ is a symmetric field, we have omitted the traditional negative sign
in front of the definition of $H_{t}^{W}$. For fixed path $b$, this
Hamiltonian $H_{t}^{W}\left(  b\right)  $ is a Gaussian random variable w.r.t
$W$.

The polymer $\tilde{P}_{b}$ based on $b$ in the random Hamiltonian $H^{W}$ is
defined as the law whose Radon-Nykodym derivative with respect to $P_{b}$ is
$Z_{t}\left(  b\right)  /E_{b}\left[  Z_{t}\left(  b\right)  \right]  $ where%
\[
Z_{t}\left(  b\right)  :=\exp H_{t}^{W}\left(  b\right)  .
\]
We use the notation $u$ for the partition function (normalizing constant) for
this measure:%
\[
u\left(  t\right)  :=E_{b}\left[  Z_{t}\left(  b\right)  \right]  .
\]
The process $u\left(  t\right)  $ is of special importance: its behavior helps
understand the behavior of the whole measure $\tilde{P}_{b}$. When $b_{0}=x$
instead of $0$, the resulting $u\left(  t,x\right)  $ is the solution of a
stochastic heat equation with multiplicative noise potential $W$, and the
logarithm of this solution solves a so-called stochastic Burgers equation.

It is known that $t^{-1}\log u\left(  t\right)  $ typically converges almost
surely to a non-random constant $\lambda$ called the almost sure Lyapunov
exponent of $u$ (see \cite{RT} and references therein for instance; the case
of random $Q$ is treated in \cite{KVV}; the case of inhomogeneous $Q$ on
compact space is discussed in \cite{FV}). The speed of concentration of $\log
u\left(  t\right)  $ around its mean has been the subject of some debate
recently. One may consult \cite{BTV} for a discussion of the issue and its
relation to the so-called wandering exponent in non-compact space. The
question is to evaluate the asymptotics of $\log u\left(  t\right)
-\mathbf{E}\left[  \log u\left(  t\right)  \right]  $ for large $t$, or to
show that it is roughly equivalent to $t^{\chi}$, where $\chi$ is called the
\emph{fluctuation exponent}. The most widely used measure of this behavior is
the asymptotics of $Var\left[  \log u\left(  t\right)  \right]  $. Here we
show that if space is compact with positive correlations, or if $W$ has
infinite spatial correlation range, then $Var\left[  \log u\left(  t\right)
\right]  $ behaves as $t$, i.e. the fluctuation exponent $\chi$ is $1/2$. This
result is highly robust to the actual distribution of $W$, since it does not
depend on the law of $Q$ under $\mathcal{P}$ beyond its first moment. We also
provide a class of examples in which $H^{W}$ is replaced by a non-linear
functional of $W$, and yet the fluctuation exponent, as measured by the power
behavior of $Var\left[  \log u\left(  t\right)  \right]  $, is still $1/2$.

We hope that our method will stimulate the study of this problem for other
correlation structures not covered by the theorem below, in particular in
infinite space when the correlation range of $W$ is finite or decaying at a
certain speed at infinity, or in the case of space-time white-noise in
discrete space, i.e. when the Brownian motions $\left\{  W\left(
\cdot,x\right)  :x\in\mathbf{Z}^{d}\right\}  $ form an IID family. We
conjecture that $\chi$ will depend on the decorrelation speed of $W$. It is at
least believed by some that in the case of space-time white noise, $\chi
<1/2$.\bigskip

The starting point for studying $Var\left[  \log u\left(  t\right)  \right]  $
is the estimation of the function $g$ relative to the random variable $\log
u\left(  t\right)  =\log E_{b}\left[  \exp H_{t}^{W}\left(  b\right)  \right]
$. Here because the integral $H_{t}^{W}\left(  b\right)  =\int_{0}^{t}W\left(
ds,b_{s}\right)  $ has to be understood as $\int_{0}^{t}\int_{\mathbf{R}%
}M\left(  ds,d\lambda\right)  \sqrt{\hat{Q}\left(  \lambda\right)
}e^{i\lambda\cdot b_{s}}$, we must calculate the Malliavin derivative with
parameters $r$ and $\lambda$. We will use the consequence of Mehler's formula
described in Lemma \ref{lemDLX} of Section \ref{PaulTools}. More specifically,
we have the following.

\begin{lemma}
\label{deewaimee}Assume $Q$ is homogeneous. Let%
\[
Y:=\frac{\log u\left(  t\right)  -\mathbf{E}\log u\left(  t\right)  }{\sqrt
{t}}.
\]
Then
\[
D_{s,\lambda}Y=\frac{1}{\sqrt{t}}\frac{1}{u\left(  t\right)  }E_{b}\left[
\sqrt{\hat{Q}\left(  \lambda\right)  }e^{i\lambda\cdot b_{s}}e^{H_{t}%
^{W}\left(  b\right)  }\right]  ,
\]
and
\begin{equation}
\left\langle DY,-DL^{-1}Y\right\rangle =\frac{1}{2t}\int_{-\pi/2}^{\pi
/2}\left\vert \sin\theta\right\vert d\theta~\mathbf{E}^{\prime}E_{b,\bar{b}%
}\left[  \int_{0}^{t}ds~Q\left(  b_{s}-\bar{b}_{s}\right)  \frac{\exp
H_{t}^{W}\left(  b\right)  }{u\left(  t\right)  }\frac{\exp H_{t}^{R_{\theta
}W}\left(  \bar{b}\right)  }{R_{\theta}u\left(  t\right)  }\right]  .
\label{deewai}%
\end{equation}
where $E_{b,\bar{b}}$ is the expectation w.r.t. two independent copies $b$ and
$\bar{b}$ of Brownian motion, and $R_{\theta}W$ was defined in (\ref{Rtheta}).
When $Q$ is inhomogeneous, the above formula still holds, with $Q\left(
b_{s}-\bar{b}_{s}\right)  $ replaced by $Q\left(  b_{s},\bar{b}_{s}\right)  $.
\end{lemma}

\begin{proof}
By formula (\ref{DM}) and the chain rule for Malliavin derivatives, we have
for fixed $b$,%
\[
D_{s,\lambda}\left(  e^{H_{t}^{W}\left(  b\right)  }\right)  =\sqrt{\hat
{Q}\left(  \lambda\right)  }e^{i\lambda\cdot b_{s}}e^{H_{t}^{W}\left(
b\right)  }%
\]
and therefore by linearity of the expectation $E_{b}$, and the chain rule
again, the first statement of the lemma follows immediately.

Now we investigate $DL^{-1}Y$. To use Lemma \ref{lemDLX} relative to $W$, we
note that the expression for $R_{\theta}Y$ is straightforward, since $Y$ is
defined as a non-random non-linear functional of an expression involving $b$
and $W$ with the latter appearing linearly via $H_{t}^{W}\left(  b\right)  $;
in other words, $R_{\theta}Y$ is obtained by replacing $H_{t}^{W}\left(
b\right)  $ by $H_{t}^{R_{\theta}W}\left(  b\right)  $, so we simply have%
\[
R_{\theta}Y=\frac{\log E_{b}\left[  \exp\left(  H_{t}^{W}\left(  b\right)
\cos\theta+H_{t}^{W^{\prime}}\left(  b\right)  \sin\theta\right)  \right]
-\mathbf{E}\log u\left(  t\right)  }{\sqrt{t}}.
\]
Thus by Lemma \ref{lemDLX},%
\[
-D_{s,\lambda}L^{-1}Y=\int_{-\pi/2}^{\pi/2}d\theta~\frac{sgn\left(
\theta\right)  }{2\sqrt{t}}E_{b}\mathbf{E}^{\prime}\left[  \sqrt{\hat
{Q}\left(  \lambda\right)  }e^{i\lambda\cdot b_{s}}\sin\left(  \theta\right)
\frac{\exp H_{t}^{R_{\theta}W}\left(  \bar{b}\right)  }{R_{\theta}u\left(
t\right)  }\right]  .
\]
We may thus calculate explicitly the inner product $\left\langle
DX,-DL^{-1}X\right\rangle $, using equation (\ref{QhatQ}), obtaining the
second announced result (\ref{deewai}). The proof of the first statement is
identical in structure to the above arguments. The last statement is obtained
again using identical arguments.
\end{proof}

\bigskip

It is worth noting that a similar expression as for $\left\langle
DY,-DL^{-1}Y\right\rangle $ can be obtained for $\left\Vert DY\right\Vert
^{2}$. Using the same calculation technique as in the above proof, we have%
\begin{align}
\left\Vert DY\right\Vert ^{2} &  =\left\Vert DY\right\Vert _{L^{2}\left(
[0,t]\times\mathbf{R}\right)  }^{2}=\frac{1}{t}E_{b,\bar{b}}\left[
\frac{e^{H_{t}^{W}\left(  b\right)  }e^{H_{t}^{W}\left(  \bar{b}\right)  }%
}{u^{2}\left(  t\right)  }\int_{0}^{t}dsQ\left(  b_{s},\bar{b}_{s}\right)
\right]  \nonumber\\
&  =\frac{1}{t}\tilde{E}_{b,\bar{b}}\left[  \int_{0}^{t}dsQ\left(  b_{s}%
,\bar{b}_{s}\right)  \right]  ,\label{overlap}%
\end{align}
where the last expression involves the expectation w.r.t. the polymer measure
$\tilde{P}$ itself, or rather w.r.t. the product measure $d\tilde{P}%
_{b,\bar{b}}=e^{H_{t}^{W}\left(  b\right)  }e^{H_{t}^{W}\left(  \bar
{b}\right)  }u^{-2}\left(  t\right)  dP_{b}\times dP_{\bar{b}}$ of two
independent polymers $\left(  b,\bar{b}\right)  $ in the same random
environment $W$. This measure is called the two-replica polymer measure, and
the quantity $\tilde{E}_{b,\bar{b}}\left[  \int_{0}^{t}dsQ\left(  b_{s}%
,\bar{b}_{s}\right)  \right]  $ is the so-called \emph{replica overlap} for
this polymer. This notion should be familiar to those studying spin glasses
such as the Sherrington-Kirkpatrick model (see \cite{TalagSpin}). The strategy
developped in this article suggests that the expression $\left\langle
DY,-DL^{-1}Y\right\rangle $ may be better suited than the rescaled overlap
$\left\Vert DY\right\Vert ^{2}$ in seeking lower bounds on $\log u$'s concentration.

\begin{notation}
\label{NotaQ}In order to simplify the notation in the next theorem, when $Q$
is not homogeneous, we denote $Q\left(  0\right)  =\max_{x\in\mathbf{R}%
}Q\left(  x,x\right)  $. We then have, in all cases, $Q\left(  0\right)
\geq\left\vert Q\left(  x,y\right)  \right\vert $ for all $x,y\in\mathbf{R}$.
Similarly we denote $Q_{m}=\min_{x,y\in\mathbf{R}}Q\left(  x,y\right)  $. In
the homogeneous case $Q_{m}$ thus coincides with $\min_{x\in\mathbf{R}%
}Q\left(  x\right)  $. When $Q$ is random, assumptions about $Q$ below are to
be understood as being required $\mathcal{P}$-almost surely.
\end{notation}

\begin{definition}
\label{DefChi}To make precise statements about the fluctiation exponent, it is
convenient to use the following definition:%
\[
\chi:=\lim_{t\rightarrow\infty}\frac{\log Var\left[  \log u\left(  t\right)
\right]  }{2\log t}%
\]

\end{definition}

\begin{theorem}
\label{thmpoly}~

\begin{enumerate}
\item Assume $Q\left(  0\right)  $ is finite. We have for all $a,t>0$,%
\begin{equation}
\mathbf{P}\left[  \left\vert \log u\left(  t\right)  -\mathbf{E}\left[  \log
u\left(  t\right)  \right]  \right\vert >a\sqrt{t}\right]  \leq1\wedge
\frac{2Q\left(  0\right)  ^{1/2}}{a\sqrt{2\pi}}\exp\left(  -\frac{a^{2}%
}{2Q\left(  0\right)  }\right)  .\label{LD}%
\end{equation}
If $Q$ is random, one only needs to take an expectation $\mathbf{E}%
_{\mathcal{P}}$ of the above right-hand side.

\item Assume $Q\left(  0\right)  $ is finite. Then for all $t$,%
\begin{equation}
Var\left[  \log u\left(  t\right)  \right]  \leq\left(  \frac{\pi}{2}\right)
^{2}\mathbf{E}_{\mathcal{P}}\left[  Q\left(  0\right)  \right]  t.
\label{LDvar}%
\end{equation}

\item Assume $Q_{m}$ is positive. Then for all $t$,%
\begin{equation}
Var\left[  \log u\left(  t\right)  \right]  \geq K_{u}\mathbf{E}_{\mathcal{P}%
}\left[  Q_{m}\right]  t, \label{LDvar2}%
\end{equation}
where the universal constant $K_{u}\simeq0.21367$ is defined in Point 1 of
Corollary \ref{thmlbcor}.

\item Assume $Q_{m}$ is positive and $Q\left(  0\right)  $ is finite. Then, in
addition to (\ref{LD}), we have for any $K<1$ and all $a$ large,
\begin{equation}
\mathbf{P}\left[  \left\vert \log u\left(  t\right)  -\mathbf{E}\left[  \log
u\left(  t\right)  \right]  \right\vert >a\sqrt{t}\right]  \geq K\frac
{Q_{m}^{1/2}}{a}\exp\left(  -\frac{a^{2}}{2Q_{m}}\right)  \label{LD2}%
\end{equation}

Moreover, the conclusions (\ref{LDvar}) and (\ref{LDvar2}) hold
simultaneously, so that the fluctuation exponent is $\chi=1/2$ as soon as
$Q\left(  0\right)  \in L^{1}\left[  \mathcal{P}\right]  $.
\end{enumerate}
\end{theorem}

The hypotheses in Points 3 and 4 of this theorem are satisfied if the state
space $\mathbf{R}$ is replaced by a compact set such as $\mathbf{S}^{1}$, or a
finite set, and $Q$ is positive everywhere: then indeed $Q_{m}>0$. Although
the hypothesis of uniform positivity of $Q$ can be considered as restrictive
for non-compact state space, one notes that there is no restriction on how
small $Q_{m}$ can be compared to $Q\left(  0\right)  $; in this sense, the
slightest persistent correlation of the random environment at distinct sites
results in a fluctuation exponent $\chi=1/2$. In sharp contrast is the case of
space-time white noise in discrete space, which is not covered by our theorem,
since then $Q\left(  x\right)  =0$ except if $x=0$; the main open problem in
discrete space is to prove that $\chi<1/2$ in this white noise case.

In relation to the overlap $\left\Vert DY\right\Vert ^{2}$, we see that under
the assumptions of Point 4 above, $\left\Vert DY\right\Vert $ is also bounded
above and below by non-random multiples of $t^{1/2}$. Hence, while our proofs
cannot use $\left\Vert DY\right\Vert $ directly to prove $\chi=1/2$, the
situation in which we can prove $\chi=1/2$ coincides with a case where the
overlap has the same rough large-time behavior as $Var\left[  \log u\left(
t\right)  \right]  $. We believe this is in accordance with common intuition
about related spin glass models.

More generally, we consider it an important open problem to understand the
precise deviations of $\log u\left(  t\right)  $. The combination of the
sub-Gaussian and super-Gaussian estimates (\ref{LD}) and (\ref{LD2}) are close
to a central limit theorem statement, except for the fact that the rate is not
sharply pinpointed. Finding a sharper rate is an arduous task which will
require a finer analysis of the expression (\ref{deewai}), and should depend
heavily and non-trivially on the correlations of the covariance function, just
as the obtention of a $\chi<1/2$ should depend on having correlations that
decay at infinity sufficiently fast. There, we believe that a fine analysis
will reveal differences between $G$ and the overlap $\left\Vert DY\right\Vert
^{2}$, so that precise quantitative asymptotics of $\log u\left(  t\right)  $
can only be understood by analyzing $G$, not merely $\left\Vert DY\right\Vert
^{2}$. For instance, it is trivial to prove that $\mathbf{E}\left[  G\right]
\leq\mathbf{E}\left[  \left\Vert DY\right\Vert ^{2}\right]  $, and we
conjecture that this inequality is asymptotically strict for large $t$, while
the deviations of $G$ and $\left\Vert DY\right\Vert ^{2}$ themselves from
their respective means are quite small, so that their means' behavior is determinant.

Answering these questions is beyond this article's scope; we plan to pursue
them actively in the future.\vspace{0.08in}

\begin{proof}
[Proof of Theorem \ref{thmpoly}]\emph{Proof of Point 1}. Since $Q\left(
x,y\right)  \leq Q\left(  0\right)  $ for all $x,y$, from Lemma
\ref{deewaimee}, we have
\[
\left\langle DY,-DL^{-1}Y\right\rangle \leq\frac{Q\left(  0\right)  }{2t}%
\int_{-\pi/2}^{\pi/2}\left\vert \sin\theta\right\vert d\theta~t~\mathbf{E}%
^{\prime}E_{b,\bar{b}}\left[  \frac{\exp H_{t}^{W}\left(  b\right)  }{u\left(
t\right)  }\frac{\exp H_{t}^{R_{\theta}W}\left(  \bar{b}\right)  }{R_{\theta
}u\left(  t\right)  }\right]  =Q\left(  0\right)  ,
\]
where we used the trivial facts that $E_{b}\left[  \exp H_{t}^{W}\left(
b\right)  \right]  =u\left(  t\right)  $ and $E_{b}\left[  \exp H_{t}%
^{R_{\theta}W}\left(  b\right)  \right]  =R_{\theta}u\left(  t\right)  $. The
upper bound result in Theorem \ref{thmlb}, applied to the random variable
$X=Y/\sqrt{Q\left(  0\right)  },$ now yields%
\[
\mathbf{P}\left[  Y>z\right]  =\mathbf{P}\left[  X>zQ\left(  0\right)
^{-1/2}\right]  \leq\left(  1+\frac{Q\left(  0\right)  }{z^{2}}\right)
\bar{\Phi}\left(  \frac{z}{Q\left(  0\right)  ^{1/2}}\right)
\]
and the upper bound statement (\ref{LD}).\vspace{0.2cm}

\emph{Proof of Points 2 and 3}. Now we note that, since all terms in the
integrals in Lemma \ref{deewaimee} are positive, our hypothesis that $Q\left(
x,y\right)  \geq Q_{m}>0$ for all $x,y$ implies%
\begin{align*}
\left\langle DX,-DL^{-1}X\right\rangle  &  \geq\frac{Q_{m}}{2t}\int_{-\pi
/2}^{\pi/2}\left\vert \sin\theta\right\vert d\theta~t~\mathbf{E}^{\prime
}E_{b,\bar{b}}\left[  \frac{\exp H_{t}^{W}\left(  b\right)  }{u\left(
t\right)  }\frac{\exp H_{t}^{R_{\theta}W}\left(  \bar{b}\right)  }{R_{\theta
}u\left(  t\right)  }\right] \\
&  =\frac{Q_{m}}{2}\int_{-\pi/2}^{\pi/2}\left\vert \sin\theta\right\vert
d\theta~\mathbf{E}^{\prime}\left[  \frac{E_{b}\left[  \exp H_{t}^{W}\left(
b\right)  \right]  }{u\left(  t\right)  }\frac{E_{b}\left[  \exp
H_{t}^{R_{\theta}W}\left(  \bar{b}\right)  \right]  }{R_{\theta}u\left(
t\right)  }\right]  =Q_{m}.
\end{align*}
Applying Point 1 in Corollary \ref{thmlbcor} to the random variable
$X=Y/\sqrt{Q_{m}}$, the lower bound of (\ref{LDvar2}) in Point 3 follows. The
upper bound (\ref{LDvar}) of Point 2 can be proved using the result (\ref{LD})
of Point 1, although one obtains a slightly larger constant than the one
announced. The constant $\left(  \pi/2\right)  ^{2}$ is obtained by using the
bound $\left\Vert DY\right\Vert ^{2}\leq Q\left(  0\right)  $ which follows
trivially from (\ref{overlap}), and then applying the classical result
$Var\left[  Y\right]  \leq\left(  \pi/2\right)  ^{2}\mathbf{E}\left[
\left\Vert DY\right\Vert ^{2}\right]  $, found for instance in \cite[Theorem
9.2.3]{U}.\vspace{0.2cm}

\emph{Proof of Point 4.} Since $Q\left(  0\right)  $ is finite and $Q_{m}$ is
positive, using $X=Y/\sqrt{Q_{m}}$ in Corollary \ref{corlblast}, we have that
$g\left(  x\right)  \geq1$ and $g\left(  x\right)  \leq Q\left(  0\right)
/Q_{m}$, so that we may use any value $c^{\prime}>0$ in the assumption of that
corollary, with thus $K=1/\left(  2c^{\prime}+1\right)  $ arbitrarily close to
$1$; the corollary's conclusion is the statement of Point 4. This finishes the
proof of the theorem.\vspace{0.3cm}
\end{proof}

\subsection{Robustness of the fluctuation exponent: a non-Gaussian
Hamiltonian}

The statements of Point 4 of Theorem \ref{thmpoly} show that if the random
environment's spatial covariance is bounded above and below by positive
constants, then the partition function's logarithm $\log u\left(  t\right)  $
is both sub-Gaussian and super-Gaussian, in terms of its tail behavior (tail
bounded respectively above and below by Gaussian tails). We now provide an
example of a polymer subject to a non-Gaussian Hamiltonian, based still on the
same random environment, whose logarithmic partition function may not be
sub-Gaussian, yet still has a fluctuation exponent equal to $1/2$. It is
legitimate to qualify the persistence of this value $1/2$ in a non-Gaussian
example as a type of robustness.

Let%
\[
X_{t}^{W}\left(  b\right)  :=\int_{0}^{t}W\left(  ds,b_{s}\right)  .
\]
With $F\left(  t,x\right)  =x+x\left\vert x\right\vert /t$, we define our new
Hamiltonian as%
\begin{equation}
H_{t}^{W}\left(  b\right)  :=F\left(  t,X_{t}^{W}\left(  b\right)  \right)
.\label{cochondelait}%
\end{equation}
Similarly to Lemma \ref{deewaimee}, using the Chain Rule for Malliavin
derivatives, we can prove that%
\begin{equation}
D_{s,\lambda}Y=\frac{1}{\sqrt{t}}\frac{1}{u\left(  t\right)  }E_{b}\left[
\sqrt{\hat{Q}\left(  \lambda\right)  }e^{i\lambda\cdot b_{s}}e^{H_{t}%
^{W}\left(  b\right)  }\left(  1+\frac{\left\vert X_{t}^{W}\left(  b\right)
\right\vert }{t}\right)  \right]  ,\label{Deedee}%
\end{equation}
and
\begin{align}
\left\langle DY,-DL^{-1}Y\right\rangle  &  =\frac{1}{2t}\int_{-\pi/2}^{\pi
/2}\left\vert \sin\theta\right\vert d\theta~\mathbf{E}^{\prime}E_{b,\bar{b}%
}\left[  \int_{0}^{t}ds~Q\left(  b_{s},\bar{b}_{s}\right)  \frac{\exp
H_{t}^{W}\left(  b\right)  }{u\left(  t\right)  }\frac{\exp H_{t}^{R_{\theta
}W}\left(  \bar{b}\right)  }{R_{\theta}u\left(  t\right)  }\right.
\nonumber\\
&  \left.  \left(  1+\frac{\left\vert X_{t}^{W}\left(  b\right)  \right\vert
}{t}\right)  \left(  1+\frac{\left\vert X_{t}^{R_{\theta}W}\left(  \bar
{b}\right)  \right\vert }{t}\right)  \right]  .\label{Deewaing}%
\end{align}

\begin{theorem}
\label{thmng}Consider $u\left(  t\right)  =E_{b}\left[  \exp H_{t}^{W}\left(
b\right)  \right]  $ where the new Hamiltonian $H_{t}^{W}$ is given in
(\ref{cochondelait}). The random environment $W$ is as it was defined in
Section \ref{RE}, and $Q\left(  0\right)  $ and $Q_{m}$ are given in Notation
\ref{NotaQ}. $K_{u}\simeq0.21367$ is defined in Point 1 of Corollary
\ref{thmlbcor}.

\begin{enumerate}
\item Assume $Q\left(  0\right)  <1/9$. Then $Var\left[  \log u\left(
t\right)  \right]  \leq2^{8}\left(  \pi/2\right)  ^{2}\mathbf{E}\left[
Q^{3}\left(  0\right)  \right]  t+o\left(  t\right)  $.

\item Assume $Q_{m}$ is positive. Then $Var\left[  \log u\left(  t\right)
\right]  \geq K_{u}\mathbf{E}_{\mathcal{P}}\left[  Q_{m}\right]  t$.
\end{enumerate}

If both assumptions of Points 1 and 2 hold, the fluctuation exponent of
Definition \ref{DefChi} is $\chi=1/2$, and the conclusion of Point 4 in
Theorem \ref{thmpoly} holds.
\end{theorem}

We suspect that the logarithmic partition function $\log u\left(  t\right)  $
given by the non-Gaussian Hamiltonian in (\ref{cochondelait}) is eminently
non-Gaussian itself; in fact, the form of its dervative in (\ref{Deedee}),
with the additional factors of the form $\left(  1+X\left(  b\right)  \right)
/t$, can presumably be compared with $Y$. We conjecture, although we are
unable to prove it, that the corresponding $g\left(  y\right)  $ grows
linearly in $y$. This would show, via Corollary \ref{corA} Point 3, that $\log
u\left(  t\right)  $ has exponential tails. Other examples of non-Gaussian
Hamiltonians can be given, using the formulation (\ref{cochondelait}) with
other functions $F$, such as $F\left(  t,x\right)  =x+x\left\vert x\right\vert
^{p}/t^{(1+p)/2}$ for $p>0$. It should be noted, however, that in our Gaussian
environment, any value $p>1$ results in a partition function $u\left(
t\right)  $ with infinite first moment, in which case the arguments we have
given above for proving that $\chi=1/2$ will not work. This does not mean that
the logarithmic partition function cannot be analyzed using finer arguments;
it can presumably be proved to be non-Gaussian with heavier-than-exponential
tails when $p>1$.\bigskip

\begin{proof}
[Proof of Theorem \ref{thmng}]Since the additional terms in (\ref{Deewaing}),
compared to Lemma \ref{deewaimee}, are factors greater than $1$, the
conclusion of Point 2 follows immediately using the proof of Points 3 and 4 of
Theorem \ref{thmpoly}.

To prove that Point 1 holds, we will use the again the classical fact
$Var\left[  Y\right]  \leq\left(  \pi/2\right)  ^{2}\mathbf{E}\left[
\left\Vert DY\right\Vert ^{2}\right]  $. Here from (\ref{Deedee}) we have
immediately%
\[
\left\Vert DY\right\Vert ^{2}\leq Q\left(  0\right)  \left(  1+E_{b}\left[
\frac{e^{H_{t}^{W}\left(  b\right)  }}{u\left(  t\right)  }\frac{\left\vert
X_{t}^{W}\left(  b\right)  \right\vert }{t}\right]  \right)  ^{2}.
\]
Therefore, to get an upper bound on the variance of $Y$ uniformly in $t$ we
only need to show that the quantity%
\[
B:=\mathbf{E}\left[  \left(  E_{b}\left[  \frac{e^{H_{t}^{W}\left(  b\right)
}}{u\left(  t\right)  }\frac{\left\vert X_{t}^{W}\left(  b\right)  \right\vert
}{t}\right]  \right)  ^{2}\right]
\]
is bounded in $t$. We see that, using Jensen's inequality w.r.t. the polymer
measure, and then w.r.t. the random medium's expectation,%
\begin{align*}
B  &  =\frac{1}{t^{2}}\mathbf{E}\left[  \left(  E_{b}\left[  \frac
{e^{H_{t}^{W}\left(  b\right)  }}{u\left(  t\right)  }\log e^{\left\vert
X_{t}^{W}\left(  b\right)  \right\vert }\right]  \right)  ^{2}\right] \\
&  \leq\frac{1}{t^{2}}\mathbf{E}\left[  \left(  \log E_{b}\left[
\frac{e^{H_{t}^{W}\left(  b\right)  +\left\vert X_{t}^{W}\left(  b\right)
\right\vert }}{u\left(  t\right)  }\right]  \right)  ^{2}\right] \\
&  \leq\frac{1}{t^{2}}\log^{2}\left(  1+\mathbf{E}\left[  E_{b}\left[
\frac{e^{H_{t}^{W}\left(  b\right)  +\left\vert X_{t}^{W}\left(  b\right)
\right\vert }}{u\left(  t\right)  }\right]  \right]  \right)  .
\end{align*}
Now we evaluate%
\begin{align*}
\mathbf{E}\left[  E_{b}\left[  \frac{e^{H_{t}^{W}\left(  b\right)  +\left\vert
X_{t}^{W}\left(  b\right)  \right\vert }}{u\left(  t\right)  }\right]
\right]   &  =E_{b}\mathbf{E}\left[  \frac{e^{X_{t}^{W}\left(  b\right)
+\left\vert X_{t}^{W}\left(  b\right)  \right\vert ^{2}/t+\left\vert X_{t}%
^{W}\left(  b\right)  \right\vert }}{u\left(  t\right)  }\right] \\
&  \leq\mathbf{E}^{1/2}\left[  u\left(  t\right)  ^{-2}\right]  E_{b}%
\mathbf{E}^{1/2}\left[  e^{4\left\vert X_{t}^{W}\left(  b\right)  \right\vert
+2\left\vert X_{t}^{W}\left(  b\right)  \right\vert ^{2}/t}\right] \\
&  \leq\mathbf{E}^{1/2}\left[  E_{b}\left[  e^{-2H_{t}^{W}\left(  b\right)
}\right]  \right]  E_{b}\mathbf{E}^{1/2}\left[  e^{4\left\vert X_{t}%
^{W}\left(  b\right)  \right\vert +2\left\vert X_{t}^{W}\left(  b\right)
\right\vert ^{2}/t}\right]
\end{align*}
The first term in the above product is actually less than the second. For the
second, we note that for any fixed $b$, the random variable $X_{t}^{W}\left(
b\right)  $ is Gaussian centered, with a variance bounded above by $Q\left(
0\right)  t\leq t/9$. Therefore we have that $\mathbf{E}\left[  e^{4\left\vert
X_{t}^{W}\left(  b\right)  \right\vert ^{2}/t}\right]  $ is bounded by the
finite universal constant $\mathbf{E}\left[  \exp\left(  4Z^{2}/9\right)
\right]  $. This proves, via another application of Schwartz's inequality,
that for some universal constant $K_{u}^{\prime}$,%
\begin{align*}
B  &  \leq\frac{1}{t^{2}}\log^{2}\left(  1+K_{u}^{\prime}\exp\left(
16Q\left(  0\right)  t\right)  \right) \\
&  \leq\frac{1+\log^{2}K_{u}^{\prime}}{t^{2}}+16^{2}Q^{2}\left(  0\right)
=2^{8}Q^{2}\left(  0\right)  +o\left(  t\right)  ,
\end{align*}
where $o\left(  t\right)  $ is non-random, proving Point 1, and the theorem.
\end{proof}

\section{Appendix}

To prove Lemma \ref{lemkey}, we begin with an intermediate result in the $n$th
Wiener chaos.

\begin{lemma}
\label{lemsko}Let $n\in\mathbf{N}$ and $f_{n}\in H^{n}$ be a symmetric
function. Let $Y\in\mathbf{D}^{1,2}$. Then%
\begin{align*}
\mathbf{E}\left[  I_{n}\left(  f_{n}\right)  Y\right]   &  =\frac{1}%
{n}\mathbf{E}\left[  \left\langle D_{\cdot}\left(  I_{n}\left(  f_{n}\right)
\right)  ;D_{\cdot}Y\right\rangle \right] \\
&  =\mathbf{E}\left[  \left\langle I_{n-1}\left(  f_{n}\left(  \star
,\cdot\right)  \right)  ;D_{\cdot}Y\right\rangle \right]  ,
\end{align*}
where we used the notation $I_{n-1}\left(  f_{n}\left(  \star,\cdot\right)
\right)  $ to denote the function $r\mapsto I_{n-1}\left(  f_{n}\left(
\star,r\right)  \right)  $ where $I_{n-1}$ operates on the $n-1$ variables
\textquotedblleft$\star$\textquotedblright\ of $f_{n}\left(  \star,r\right)  $.
\end{lemma}

\begin{proof}
This is an immediate consequence of formula (\ref{DInfn}) and the famous
relation $\delta D=-L$ (where $\delta$ is the divergence operator (Skorohod
integral), adjoint of $D$, see \cite[Proposition 1.4.3]{Nbook}).

Here, however, we present a direct proof. Note that, because of the Wiener
chaos expansion of $Y$ in Proposition \ref{orthog}, and the fact that all
chaos terms of different orders are orthogonal, without loss of generality, we
can assume $Y=I_{n}\left(  g_{n}\right)  $ for some symmetric $g_{n}\in H^{n}%
$; then, using the formula for the covariance of two $n$th-chaos r.v.'s in
Proposition \ref{orthog}, we have%
\begin{align*}
\mathbf{E}\left[  \left\langle I_{n-1}\left(  f_{n}\left(  \star,\cdot\right)
\right)  ;D_{\cdot}Y\right\rangle \right]   &  =\mathbf{E}\left[  \left\langle
I_{n-1}\left(  f_{n}\left(  \star,\cdot\right)  \right)  ;nI_{n-1}\left(
g_{n}\left(  \star,\cdot\right)  \right)  \right\rangle \right] \\
&  =n\int_{T}\mathbf{E}\left[  I_{n-1}\left(  f_{n}\left(  \star,r\right)
\right)  I_{n-1}\left(  g_{n}\left(  \star,r\right)  \right)  \right]
\mu\left(  dr\right) \\
&  =n\int_{T}\left(  n-1\right)  !\left\langle f_{n}\left(  \star,r\right)
,g_{n}\left(  \star,r\right)  \right\rangle _{L^{2}\left(  T^{n-1}%
,\mu^{\otimes n-1}\right)  }\mu\left(  dr\right) \\
&  =n!\left\langle f_{n};g_{n}\right\rangle _{L^{2}\left(  T^{n},\mu^{\otimes
n}\right)  }=\mathbf{E}\left[  I_{n}\left(  f_{n}\right)  Y\right]  .
\end{align*}
which, together with formula (\ref{DInfn}), proves the lemma.\bigskip
\end{proof}

\begin{proof}
[Proof of Lemma \ref{lemkey}]Since $X\in\mathbf{D}^{1,2}$ and is centered, it
has a Wiener chaos expansion $X=\sum_{n=1}^{\infty}I_{n}\left(  f_{n}\right)
$. We calculate $\mathbf{E}\left[  Xh\left(  X\right)  \right]  $ via this
expansion and the Malliavin calculus, invoking Remark \ref{remgenchain} and
using Lemma \ref{lemsko}:%
\begin{align*}
\mathbf{E}\left[  Xh\left(  X\right)  \right]   &  =\sum_{n=1}^{\infty
}\mathbf{E}\left[  I_{n}\left(  f_{n}\right)  h\left(  X\right)  \right]  \\
&  =\sum_{n=1}^{\infty}\frac{1}{n}\mathbf{E}\left[  \int_{T}D_{r}I_{n}\left(
f_{n}\right)  \ D_{r}h\left(  X\right)  \ \mu\left(  dr\right)  \right]  \\
&  =\mathbf{E}\left[  h^{\prime}\left(  X\right)  \int_{T}D_{r}\left(
\sum_{n=1}^{\infty}\frac{1}{n}I_{n}\left(  f_{n}\right)  \right)
\ D_{r}X\ \mu\left(  dr\right)  \right]
\end{align*}
which by the definition of $-L$ is precisely the statement (\ref{easykey}%
).\bigskip
\end{proof}

\begin{proof}
[Proof of Lemma \ref{lemDLX}]The proof goes exactly as that of Lemma 1.5.2 in
\cite{Nbook}, with only the following change: the point there was to represent
the Malliavin derivative of the operator $\left(  -C\right)  ^{-1}$ which
changes $I_{n}\left(  f_{n}\right)  $ into $n^{-1/2}I_{n}\left(  f_{n}\right)
$, whereas here $-L^{-1}$ changes $I_{n}\left(  f_{n}\right)  $ into
$n^{-1}I_{n}\left(  f_{n}\right)  $; in \cite[Lemma 1.5.2]{Nbook}, the
function $\varphi$ was introduced with the property that $\int_{-\pi/2}%
^{\pi/2}\sin\left(  \mathbf{\theta}\right)  \cos^{n}\left(  \theta\right)
\varphi\left(  \theta\right)  d\theta=\left(  n+1\right)  ^{-1/2}$; here we
therefore only need to replace that $\varphi$ by a function $\phi$ such that
$\int_{-\pi/2}^{\pi/2}\sin\left(  \mathbf{\theta}\right)  \cos^{n}\left(
\theta\right)  \phi\left(  \theta\right)  d\theta=\left(  n+1\right)  ^{-1}$.
It is clear that this function $\phi$ is $\phi\left(  \theta\right)
=2^{-1}sgn\left(  \theta\right)  $. Our lemma follows by the proof of
\cite[Lemma 1.5.2]{Nbook}.
\end{proof}

For completeness, we finish with a short, self-contained proof of the upper
bound in Theorem \ref{thm12}, which is equivalent to Proposition
\ref{fundu}.\vspace{0.3cm}

\begin{proof}
[Proof of Proposition \ref{fundu}]Assume $X$ is centered and $W$ is the
standard Wiener space. By the Clark-Ocone representation formula
(\ref{clarkocone}), we can define a continuous square-integrable martingale
$M$ with $M\left(  1\right)  =X$, via the formula $M\left(  t\right)
:=\int_{0}^{t}\mathbf{E}\left[  D_{s}X|\mathcal{F}_{s}\right]  dW\left(
s\right)  $. The quadratic variation of $M$ is equal to $\left\langle
M\right\rangle _{t}=\int_{0}^{t}\left\vert \mathbf{E}\left[  D_{s}%
X|\mathcal{F}_{s}\right]  \right\vert ^{2}$; therefore, by hypothesis,
$\left\langle M\right\rangle _{t}\leq t$. Using the Doleans-Dade exponential
martingale $\mathcal{E}\left(  \lambda M\right)  $ based on $\lambda M$,
defined by $\mathcal{E}\left(  \lambda M\right)  _{t}=\exp\left(  \lambda
M_{t}-\frac{\lambda^{2}}{2}\left\langle M\right\rangle _{t}\right)  $ we now
have
\[
\mathbf{E}\left[  \exp\lambda X\right]  =\mathbf{E}\left[  \mathcal{E}\left(
\lambda M\right)  _{1}\exp\left(  \frac{\lambda^{2}}{2}\left\langle
M\right\rangle _{1}\right)  \right]  \leq\mathbf{E}\left[  \mathcal{E}\left(
\lambda M\right)  _{1}\right]  e^{\lambda^{2}/2}=e^{\lambda^{2}/2}.
\]
The proposition follows using a standard optimization calculation and
Chebyshev's inequality. Theorem 9.1.1 in \cite{U} can be invoked to prove the
same estimate in the case of a general isonormal Gaussian process $W$.
\end{proof}

\end{document}